\documentclass[11pt,reqno]{amsart}
\usepackage{amsmath}
\usepackage{amssymb}
\usepackage{amsthm}
\usepackage{color}
\usepackage{eucal}
\usepackage{tikz}
\usepackage{gastex}
\usetikzlibrary{decorations.pathmorphing}
\DeclareSymbolFont{rsfscript}{OMS}{rsfs}{m}{n}
\DeclareSymbolFontAlphabet{\mathrsfs}{rsfscript}

\def\softd{{\leavevmode\setbox1=\hbox{d}\hbox
            to 1.15\wd1{d\kern-0.2ex{\char039}\hss}}}    
\def\softl{l\kern-0.3ex\raise0.1ex\hbox{'}\kern-0.3ex}   
\numberwithin{equation}{section}
\newtheorem{Thm}{Theorem}[section]
\newtheorem{Prop}[Thm]{Proposition}
\newtheorem{Lemma}[Thm]{Lemma}

\newtheorem{Cor}[Thm]{Corollary}

\theoremstyle{remark}
\def\rrbb{{\!-\!\!\!-\!\!\!-\!}}

\def\Dc{\mathrel{\mathrsfs{D}}}

\def\Lc{\mathrel{\mathrsfs{L}}}
\def\Rc{\mathrel{\mathrsfs{R}}}
\def\A{\mathfrak{A}}
\def\C{\mathfrak{C}}
\def\B{\mathfrak{B}}

\def\T{\mathfrak{T}}
\def\S{\mathfrak{S}}
\def\R{\mathfrak{R}}
\def\L{\mathfrak{L}}

\def\si{\sigma}
\def\wt{\widetilde}
\def\Ga{\Gamma}
\def\De{\Delta}
\def\Te{\Theta}
\def\co{\mathrm{c}}
\def\r{\mathrm{r}}
\def\l{\mathrm{l}}

\def\la{\lambda}
\def\al{\alpha}
\def\be{\beta}
\def\ga{\gamma}
\def\de{\delta}
\def\lv{\mathfrak{l}}
\def\rv{\mathfrak{r}}
\def\sv{\mathfrak{s}}
\def\cb{\mathbf{c}}
\def\sk{\mathsf{sk}}
\def\Ls{\mathsf{L}}
\def\Rs{\mathsf{R}}

\def\ol{\overline}
\def\wt{\widetilde}
\def\wh{\widehat}
\newcommand{\pseudo}{pseudosemilattice}

\def\circledm{\protect\mathbin{\hbox
    {\protect$\bigcirc$\rlap{\kern-8.2pt\raise0pt\hbox
    {\protect$\mathtt{m}$}}}}}        
\def\smallcircledm{\protect\mathbin{\hbox
    {\protect$\bigcirc$\rlap{\kern-6.9pt\raise0pt\hbox
    {\protect$\mathtt{m}$}}}}}

\def\wr{\mathrel{{\le}_{\Rc}}}
\def\wl{\mathrel{{\le}_{\Lc}}}
\title[]{On the variety of strict pseudosemilattices}

\author{K.~Auinger}
\address{Fakult\"at f\"ur Mathematik, Universit\"at Wien,
Nordbergstrasse 15, A-1090 Wien, Austria}
\email{karl.auinger@univie.ac.at}
\author{L.~Oliveira}
\address{Departamento de Matem\'atica Pura,
Faculdade de Ci\^encias da Universidade do Porto,
R. Campo Alegre, 687, 4169-007 Porto, Portugal}
\email{loliveir@fc.up.pt}
\begin{document}

\begin{abstract} A
new model, in terms of finite bipartite graphs, of the free pseudosemilattice is presented. This will then be used to obtain several results about the variety $\mathbf{SPS}$ of all strict pseudosemilattices: (i) an identity basis for $\mathbf{SPS}$ is found, (ii) $\mathbf{SPS}$ is shown to be inherently non-finitely based, (iii) $\mathbf{SPS}$ is shown to have no irredundant identity basis, and (iv) $\mathbf{SPS}$ is shown to have no covers and to be $\cap$-prime in the lattice of all varieties of pseudosemilattices.
Some applications to e-varieties of locally inverse semigroups are also derived.
\end{abstract}

\maketitle
\section{Introduction} A semigroup $S$ is regular if for every element $x\in S$ there is an element $x'\in S$ such that $xx'x=x$. On the set $E(S)$ of idempotents of a regular semigroup $S$ we shall consider the following two binary relations:
\[e\,\wr f \,\Leftrightarrow\, e=fe\quad\mbox{ and }\quad e\,\wl f \,\Leftrightarrow\, e=ef.\]
Let also ${\le} ={\wr}\cap{\wl}$. Then $\wr$ and $\wl$ are quasi-orders on $E(S)$, while $\le$ is \emph{the natural partial order} on $E(S)$. We shall write also $f\mathrel{{\ge}_{\Rc}}e$, $f\mathrel{{\ge}_{\Lc}}e$
and $f\ge e$ for $e\,\wr f$, $e\,\wl f$ and $e\le f$, respectively, and denote by $(f]_{\Rc}$ the set of idempotents $e$ such that $e\,\wr f$. Similarly, we define $(f]_{\Lc}$ and $(f]_{\le}$.

\emph{Locally inverse} semigroups can be defined as regular semigroups for which any two idempotents $e$ and $f$ have another idempotent $g$ such that $(e]_{\Rc}\cap (f]_{\Lc}=(g]_{\le}$. This idempotent $g$ is unique for each ordered pair of idempotents $(e,f)$, and shall be denoted by $e\wedge f$. Thus, any locally inverse semigroup $S$ originates a new binary algebra $(E(S),\wedge)$ called the \emph{ \pseudo\ of idempotents} of $S$. The class of all such binary algebras forms a variety $\mathbf{PS}$ which can be defined by the following three identities together with the left-right duals (PS2') and (PS3') of (PS2) and (PS3) (Nambooripad \cite{na1}):
\begin{itemize}
\item[(PS1)] $x\wedge x\approx x$;
\item[(PS2)] $(x\wedge y)\wedge (x\wedge z)\approx
(x\wedge y)\wedge z$;
\item[(PS3)] $((x\wedge y)\wedge (x\wedge z))\wedge (x\wedge w)\approx (x\wedge y)
\wedge ((x\wedge z)\wedge (x\wedge w))$.
\end{itemize}
Abstractly, a \pseudo\ is a binary algebra $(E,\wedge)$ satisfying these five identities, and, as such, every \pseudo\ is the \pseudo\ of idempotents of some locally inverse semigroup. The relations ${\le}_{\Rc}$ and ${\le}_{\Lc}$ can be recovered from the operation $\wedge$, namely by
\[e\wr f\Leftrightarrow f\wedge e=e\mbox{ and } e\wl f\Leftrightarrow e\wedge f=e.\]
{\parindent=0pt In particular, the relations $e\mathrel{{\ge}_{\Rc}}e\wedge f\wl f$
are satisfied for all $e,f\in E$.}

An e-variety of regular semigroups is a class of such semigroups
closed under the formation of homomorphic images, regular subsemigroups and direct products (see \cite{ha1, ks1}). The class {\bf LI} of all locally inverse semigroups is an example of an e-variety of regular semigroups. The first author \cite{au4} showed that the mapping
\begin{equation}\label{sgtops}\varphi: {\mathcal{L}}_e(\mathbf{ LI})\longrightarrow {\mathcal{L}}(\mathbf{PS}),\quad \mathbf{V}\longmapsto \{(E(S),\wedge)\,|\, S\in \mathbf{V}\}\end{equation}
 is a well-defined complete
homomorphism from the lattice ${\mathcal{L}}_e({\bf LI})$ of e-varie\-ties of locally
inverse semigroups onto the lattice ${\mathcal{L}}(\mathbf{PS})$ of varieties of \pseudo s.

It is well known that the pseudosemilattice $(E(S),\wedge)$ of a locally inverse semigroup $S$ is associative if and only if $S$ is $E$-solid \cite[Theorem 4.1]{biequationaltheory}; the identities (PS2) and (PS2') imply that an associative pseudosemilattice is a normal band. It follows that the homomorphism  $\varphi$ of (\ref{sgtops}) maps the interval $[\mathbf{T},\mathbf{ESLI}]$ of $\mathcal{L}_e(\mathbf{ LI})$, consisting of all $E$-solid locally inverse e-varieties, onto the eight-element lattice $[\mathbf{T},\mathbf{NB}]$ consisting of all varieties of normal bands ($\mathbf{T}$ denotes here the trivial e-variety as well as the trivial variety). By Hall \cite[Theorem 3.5]{hallCSR} there exists a unique least locally inverse e-variety which is not $E$-solid: this is the e-variety $\mathbf{CSR}$ of all \emph{combinatorial strict regular semigroups} which is generated by the five-element combinatorial non-orthodox completely $0$-simple semigroup $A_2$. It follows that $\mathbf{SPS}:=\mathbf{CSR}\varphi$ is the unique least non-associative variety of pseudosemilattices the members of which are called \emph{strict pseudosemilattices}. In particular, $\mathbf{SPS}$ is generated by each of its non-associative members, an example of which is $E_2:=E(A_2)$, the pseudosemilattice of idempotents of $A_2$. It is easily seen that $\mathbf{SPS}$ (since it contains $E_2$) contains all left zero semigroups, all right zero semigroup and all semilattices and therefore contains all normal bands. In particular,  ${\mathcal{L}}(\mathbf{PS})$ is the disjoint union of the intervals $[\mathbf{T},\mathbf{NB}]$ and $[\mathbf{SPS},\mathbf{PS}]$ and $\mathbf{NB}\subseteq \mathbf{SPS}$.

The present paper has two main objectives. The second one is to construct an (infinite) identity basis for $\mathbf{SPS}$ and to prove some remarkable properties of this variety: it does not have an irredundant identity basis, it is inherently non-finitely based and has no cover in the lattice of varieties of pseudosemilattices. Some applications to e-varieties of locally inverse semigroups are also given. This will be dealt with in sections \ref{basisSPS} -- \ref{supplements}.

The aforementioned results will be obtained with the help of a new model of the free pseudosemilattice which is in terms of finite bipartite graphs and which seems to be much more transparent than the previously discovered ones. In the literature, there exist already three different models of the free pseudosemilattice. The first one is by Meakin \cite{mea2} (the two generator case is treated in \cite{mp2}), the second one by the first author \cite{aufreepseudo} and the third one by the second author \cite{l1}. There are two further models by the second author \cite{l2} which may be seen as offshoots of \cite{l1}. All these models are quite complicated and involved, the complete statement of their definitions needs quite a bit of space (therefore the authors refrain from recalling these here). Moreover,  it is by far not obvious that these models are actually models of the same structure. In fact, a direct proof showing that these structures are isomorphic is not known and seems to be tedious. For the authors it turned out to be shorter and more effective to give a direct proof that the new construction is indeed a model of the free pseudosemilattice. This has the additional advantage that the proof  does not depend on any previous result, hence is self-contained and  its understanding does not require any semigroup theoretic background. The presentation of this proof is the first objective of the paper which will be accomplished in section \ref{freePS}; some preliminaries will be collected in section \ref{algebraB(X)}.

\section{The binary algebra $\B(X)$}\label{algebraB(X)} The members of $F_2(X)$, the free binary algebra  on a non-empty set $X$ of variables (or letters) are usually written as well-formed words over the alphabet $X\cup\{(,),\wedge\}$ where $\wedge$ is a symbol for the binary operation. However, they may also
be conveniently represented by finite rooted binary trees the leaves of which are labeled by the letters of $X$, see \cite{l1}. We define this representation inductively by setting $\Ga(x)=\underset{x}{\bullet}$ for each $x\in X$ and letting, for $u,v\in F_2(X)$
 $$\tikz[scale=1]{\draw(0,0) node{$\Ga(u\wedge v):=$};\draw(1.5,-0.5)node{$\Ga(u)$};\draw(2.5,-0.5)node{$\Ga(v)$};\draw(2,0.5)node{$\bullet$};
 \draw (1.5,-0.3)--(2,0.5);\draw (2,0.5)--(2.5,-0.3);\filldraw (2.9,-0.05)circle (0.5pt);}
$$
The set $\Ga(F_2(X))=:\T(X)$ obtained this way comprises the set of all finite binary rooted trees in which each vertex except the root has a unique predecessor, each vertex not being a leaf has two successors, a \emph{left} one and a \emph{right} one (these vertices will be referred to as the \emph{left vertices} and the \emph{right vertices}, respectively), and each leaf carries a label from $X$. The vertices of $\Ga(u)$ represent uniquely determined subwords of $u$. One has an obvious binary operation $\wedge$ on $\T(X)$, namely $(\Ga(u),\Ga(v))\mapsto \Ga(u)\wedge \Ga(v):=\Ga(u\wedge v)$ so that the mapping $F_2(X)\to \T(X)$, $u\mapsto \Ga(u)$ is an isomorphism.

We define some combinatorial invariants of the members of $F_2(X)$. The \emph{content} $\co(u)$ is the set of all variables (letters) occurring in $u$; $\l(u)$ and $\r(u)$ are, respectively, the \emph{leftmost} and the \emph{rightmost} letter occurring in $u$. We  need two further, less common invariants: the \emph{left content} $\co_l(u)$ is the set of all letters that label a left leaf of $\Ga(u)$ while the \emph{right content} $\co_r(u)$ is the set of all letters that label a right leaf of $\Ga(u)$. In other words, $\co_l(x)=\co_r(x)=\emptyset$ for each letter $x$, and a letter $x$ belongs to the left content (respectively right content) of $u\in F_2(X)\setminus X$ if and only if $x\wedge t$ (respectively $t\wedge x$) occurs as a subword of $u$ for some $t\in F_2(X)$.

\subsection{The binary algebra $\B(X)$} We are going to introduce another $X$-generated binary algebra $\B(X)$ as follows. Let first $\B'(X)$ be the set of all finite non-trivial trees $\ga$ all of whose vertices are labeled by letters of $X$ and such that an ordered pair $(\lv_\ga,\rv_\ga)$ of two distinct adjacent vertices is distinguished, the \emph{left root} and the \emph{right root}, respectively, of $\ga$. The set of all vertices of $\ga$ then is naturally partitioned into two disjoint subsets $L_\ga$ and $R_\ga$ (the \emph{left vertices} and the \emph{right vertices}, respectively) by letting $L_\ga$ be the set of all vertices having even distance to $\lv_\ga$ and $R_\ga$ the set of all vertices having odd distance to $\lv_\ga$ (the definition is dual with respect to $\rv_\ga$). So, the members of $\B'(X)$ are always considered and viewed as bipartite graphs. For a left vertex $a$ and a right vertex $b$ which are adjacent the unique edge connecting $a$ and $b$ will often be denoted by the ordered pair $(a,b)$. We may define also the combinatorial invariants $\l,\r,\co,\co_l,\co_r$ for elements $\ga$ of $\B'(X)$ in a natural way --- for example $\l(\ga)$ is the label of $\lv_\ga$, $\co_r(\ga)$ is the set of all labels of all right vertices of $\ga$, and so on. In a graphical representation of an object of $\B'(X)$ it will always be assumed that the vertices are arranged either as two vertical columns (the left/right column representing the left/right vertices) or as two horizontal rows (the bottom/top row representing the left/right vertices). Moreover, in such a graphical representation  it is often convenient to indicate the distinguished left and right roots by especially indicating the unique edge connecting these vertices (for example by a double line \tikz[baseline=0.15cm,scale=0.5]{\draw[double](0,0.5)--(1.5,0.5);}).

Next let $\B(X):=\B'(X)\cup X$ where each letter $x\in X$ is now represented by the single vertex graph $\underset{x}{\bullet}$ having the label $x$ (for $\ga=\underset{x}{\bullet}$ we assume that $\lv_\ga=\rv_\ga=\ga$).
For $\ga\in \B'(X)$ we let ${}^L\ga$ be the graph $\ga$ in which the distinguished right root is unmarked; that is, ${}^L\ga$ is now the same tree as $\ga$ but only the single root $\lv_\ga$ is distinguished --- however it is distinguished as a \textbf{left} root, that is, the single distinguished root determines the partition of the set of all vertices into the set of left vertices $L_\gamma$ and the set of right vertices $R_\gamma$. One sided distinguished vertices shall usually be represented as ``encircled bullets'' like this \tikz[baseline=-4pt,scale=0.3]{\coordinate  (1) at (0,0);
 \draw (1) node{$\bullet$}; \draw (1)circle (10pt);}.
For $\ga=\underset{x}{\bullet}$ we let ${}^L\ga$ coincide with $\ga$, however, in this case $\ga$ is viewed as a bipartite graph having one left vertex and no right vertices. Define $\ga^R$ dually.

Now we introduce a binary operation $\sqcap$ on $\B(X)$ by
$$\al\sqcap\be:= {}^L\al\mathrel{\dot{\cup}}\{(\lv_\al,\rv_\be)\}\mathrel{\dot{\cup}}\be^R.$$
That is, in order to get $\al\sqcap\be$ from $\al,\be\in \B(X)$, form the disjoint union of $\al$ and $\be$, declare the distinguished left vertex of $\al$ to be that of $\al\sqcap\be$, the distinguished right vertex of $\be$ to be that of $\al\sqcap \be$ and connect these two vertices by a new edge.

An easy inductive argument on the number of vertices shows that $\B(X)$ is an $X$-generated binary algebra.  Denote by $\De$ the unique homomorphism $F_2(X)\to \B(X)$ satisfying $\De(x)=\underset{x}{\bullet}$ for each $x\in X$. In particular, for each $\ga\in \B(X)$ there exists $u\in F_2(X)$ such that $\ga=\De(u)$. There is a unique homomorphism $\chi:\T(X)\to \B(X)$ satisfying $\De=\chi\circ \Ga$. It is essential for the understanding of many arguments in this paper that the assignment $\Ga(u)\mapsto \De(u)$ can itself be viewed as being induced by a quotient map from $\Ga(u)$ onto $\De(u)$.

We claim that $\De(u)$ can be obtained from $\Ga(u)$ as follows: the sets of left/right vertices of $\De(u)$ are in bijective correspondence with the sets of left/right leaves of $\Ga(u)$. Moreover, each left/right vertex of $\De(u)$ is obtained by contracting in $\Ga(u)$ a subgraph of the form $\bullet\rrbb\!\bullet\!\rrbb\  \dots\  \rrbb\!\bullet$ (a segment) containing exactly one leaf to a vertex. More precisely, this is done as follows. Denote, for two (not necessarily distinct) vertices $a$ and $b$ of $\Ga(u)$ the unique geodesic subgraph of $\Ga(u)$ starting at $a$ and ending at $b$ by $[a,b]$. For each left leaf $a$ let now $a'$ be the unique (left) vertex such that $[a,a']$ contains only left vertices and $[a,a']$ is maximal with this property, and proceed dually with each right
leaf $b$. (The root of $\Ga(u)$ is neither a left nor a right vertex.) In order to get the left/right vertices of $\De(u)$ now contract the segments of the form $[a,a']$ with $a$ a left leaf to a (left) vertex and each segment of the form $[b',b]$ with $b$ a right leaf to a (right) vertex. The left vertex corresponding to $[a,a']$ and the right vertex corresponding to $[b',b]$ then shall be connected by an edge in $\De(u)$ if $a'$ is connected by an edge in $\Ga(u)$ with a vertex in $[b',b]$, or $b'$ is connected by an
edge in $\Ga(u)$ with a vertex in $[a,a']$; or, equivalently,
if an edge of $\Ga(u)$ connects some vertex of $[a,a']$ with some of $[b',b]$.
In addition, if $a$ is the leftmost (left) leaf of $\Ga(u)$ and $b$ is the rightmost (right) leaf of $\Ga(u)$ then the corresponding vertices in $\De(u)$ shall be also connected by an edge and shall become the distinguished left and right vertices in $\De(u)$. The latter edge can be viewed as being obtained by contracting the segment $\rrbb\!\!\bullet\!\!\rrbb$ (with $\bullet$ the root of $\Ga(u)$) to an edge. Altogether, the so described quotient map $\chi_u:\Ga(u)\to \De(u)$ is not a graph homomorphism in the usual sense but it is `almost' a  ``contraction of a family of subtrees'' in the sense of Serre \cite{trees}. A formal proof of the mentioned nature of the mapping $\chi_u:\Ga(u)\to \De(u)$ can be done by induction on the number of leaves of $\Ga(w)=\Ga(u\wedge v)$ by taking into account the following observation: if $a$ denotes the leftmost leaf of $\Ga(u)$ and the segment $[a,a']$ is replaced by $[a,0]$ where $0$ denotes the root of $\Ga(u)$ then the left-rooted bipartite tree obtained by contracting segments and letting the distinguished left vertex correspond to the segment $[a,0]$ is exactly the left-rooted tree ${}^L\De(u)$  and dually, if for the rightmost leaf $b$ of $\Ga(v)$ the segment $[b',b]$ is replaced with $[0,b]$ one gets by the analogous procedure the right-rooted tree $\De(v)^R$.

Let us consider the example
$$u=((x\wedge(v\wedge z))\wedge x)\wedge((v\wedge z)\wedge((v\wedge w)\wedge y))$$
for $v,w,x,y,z\in X$. The graph $\Ga(u)$ with the segments $[a,a']$ and $[b',b]$ already indicated looks like this:

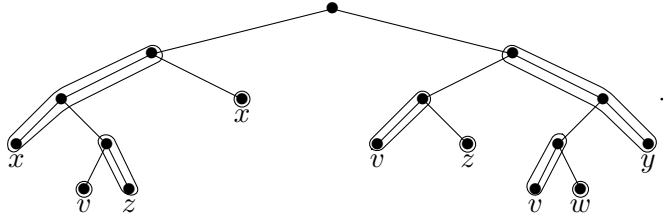
\centerline{
\tikz[scale=0.6]{
\coordinate (1) at (0,4);
\coordinate (2) at (-4,3);
\coordinate (3) at (4,3);
\coordinate (4) at (-6,2);
\coordinate[label=below:$x$] (5) at (-2,2);
\coordinate (6) at  (2,2);
\coordinate (7) at (6,2);
\coordinate[label=below:$x$] (8) at (-7,1);
\coordinate (9) at (-5,1);
\coordinate[label=below:$v$] (10) at (1,1);
\coordinate[label=below:$z$] (11) at (3,1);
\coordinate (12) at (5,1);
\coordinate[label=below:$y$] (13) at (7,1);
\coordinate[label=below:$v$] (14) at (-5.5,0);
\coordinate[label=below:$z$] (15) at (-4.5,0);
\coordinate[label=below:$v$] (16) at (4.5,0);
\coordinate[label=below:$w$] (17) at (5.5,0);
\draw (1) node {$\bullet$};
\draw (2) node {$\bullet$};
\draw (3) node {$\bullet$};
\draw (4) node {$\bullet$};
\draw (5) node {$\bullet$};
\draw (6) node {$\bullet$};
\draw (7) node {$\bullet$};
\draw (8) node {$\bullet$};
\draw (9) node {$\bullet$};
\draw (10) node {$\bullet$};
\draw (11) node {$\bullet$};
\draw (12) node {$\bullet$};
\draw (13) node {$\bullet$};
\draw (14) node {$\bullet$};
\draw (15) node {$\bullet$};
\draw (16) node {$\bullet$};
\draw (17) node {$\bullet$};
\draw (1)--(2)--(4)--(8);
\draw (2)--(5);
\draw (4)--(9)--(15);
\draw (9)--(14);
\draw (1)--(3)--(7)--(13);
\draw (3)--(6)--(10);
\draw (6)--(11);
\draw (7)--(12)--(16);
\draw (12)--(17);
\draw[rounded corners=3pt] (-3.9,3.25)--(-6.1,2.2)--(-7.25,1.05)--(-7.1,0.75)--(-5.9,1.8)--(-3.7,2.9)--cycle;
\draw[rounded corners=3pt] (-4.9,1.25)--(-4.25,-0.1)--(-4.6,-0.15)--(-5.25,1.15)--cycle;
\draw[rounded corners=3pt] (2,2.25)--(2.25,2)--(1,0.79)--(0.75,1)--cycle;
\draw[rounded corners=3pt] (4.95,1.29)--(5.25,1.15)--(4.55,-0.2)--(4.25,0)--cycle;
\draw[rounded corners=3pt] (3.9,3.25)--(6.1,2.2)--(7.25,1.05)--(6.99,0.75)--(5.9,1.8)--(3.7,2.9)--cycle;
\draw (5)circle (5pt);
\draw (11)circle (5pt);
\draw (14)circle (5pt);
\draw (17)circle (5pt);
\filldraw (7.3,2)circle (0.5pt);
}
}

{\parindent=0pt Contraction of the segments to vertices gives us the following `almost' bipartite graph}

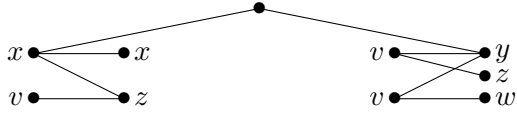
\centerline{
\tikz[scale=0.6]{
\coordinate[label=left:$v$] (1) at (-5,0);
\coordinate[label=right:$z$] (2) at (-3,0);
\coordinate[label=left:$v$] (3) at (3,0);
\coordinate[label=right:$w$] (4) at (5,0);
\coordinate[label=left:$x$] (5) at (-5,1);
\coordinate[label=right:$x$] (6) at  (-3,1);
\coordinate[label=left:$v$] (7) at (3,1);
\coordinate[label=right:$y$] (8) at (5,1);
\coordinate[label=right:$z$] (9) at (5,0.5);
\coordinate (10) at (0,2);
\draw (1) node {$\bullet$};
\draw (2) node {$\bullet$};
\draw (3) node {$\bullet$};
\draw (4) node {$\bullet$};
\draw (5) node {$\bullet$};
\draw (6) node {$\bullet$};
\draw (7) node {$\bullet$};
\draw (8) node {$\bullet$};
\draw (9) node {$\bullet$};
\draw (10) node {$\bullet$};
\draw (1)--(2)--(5)--(10)--(8)--(3)--(4);
\draw (5)--(6);
\draw (9)--(7)--(8);
}
}

{\parindent=0pt and marking the left/right roots, that is, contracting  the segment $\rrbb\!\!\bullet\!\!\rrbb$ to the edge with distinguished endpoints finally yields $\De(u)$:}

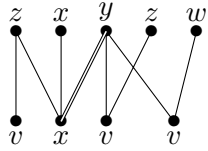
\centerline{
\tikz[scale=0.6]{
\coordinate[label=below:$v$] (1) at (0,0);
\coordinate[label=below:$x$] (2) at (1,0);
\coordinate[label=below:$v$] (3) at (2,0);
\coordinate[label=below:$v$] (4) at (3.5,0);
\coordinate[label=above:$z$] (5) at (0,2);
\coordinate[label=above:$x$] (6) at (1,2);
\coordinate[label=above:$z$] (7) at (3,2);
\coordinate[label=above:$y$] (8) at (2,2);
\coordinate[label=above:$w$] (9) at (4,2);
\draw (1) node {$\bullet$};
\draw (2) node {$\bullet$};
\draw (3) node {$\bullet$};
\draw (4) node {$\bullet$};
\draw (5) node {$\bullet$};
\draw (6) node {$\bullet$};
\draw (7) node {$\bullet$};
\draw (8) node {$\bullet$};
\draw (9) node {$\bullet$};
\draw (1)--(5)--(2)--(6);
\draw (7)--(3)--(8)--(4)--(9);
\draw[double](2)--(8);
\filldraw (4.3,1)circle (0.5pt);}
}

\subsection{Description of $\De(u(s\to t))$ and of $\De(u\psi)$}
The quotient  map $\chi_u:\Ga(u)\to \De(u)$ leads to a better understanding of the following construction. Let $u,s,t\in F_2(X)$ with $s$ a subword of $u$; let $u(s\to t)$ be the word obtained from $u$ by replacing  (a particular, earlier chosen occurrence of) the subword $s$ by $t$. We shall describe $\De(u(s\to t))$ in terms of $\De(u)$, $\De(s)$ and $\De(t)$.

First of all, recall that each vertex in $\Ga(u)$ corresponds to a unique subword of $u$. Take the vertex $a$, say, that corresponds to the occurrence of $s$ in $u$ that should be replaced. The binary tree that is formed from the vertex $a$ as root together with all vertices that can be reached by going downwards from $a$ is just $\Ga(s)$. It is obvious how $\Ga(u(s\to t))$ is formed: just replace in $\Ga(u)$ the subtree $\Ga(s)$ by $\Ga(t)$. Now, looking at the quotient mapping $\chi_u:\Ga(u)\to \De(u)$ we see that $\chi_u(\Ga(s))$ forms a subtree of $\De(u)$, namely the subtree of $\De(u)$ spanned by all vertices that are in the image under $\chi_u$ of all leaves of $\Ga(s)$. Let us consider $\chi_u(\Ga(s))$ as a bipartite graph with one distinguished vertex as follows: let the distinguished vertex of $\chi_u(\Ga(s))$ be the image of the leftmost leaf of $\Ga(s)$ if the root of $\Ga(s)$ is a left vertex in $\Ga(u)$ and the image of the rightmost leaf of $\Ga(s)$ if the root of $\Ga(s)$ is a right vertex of $\Ga(u)$ (we need only consider proper subwords $s$ of $u$, so the root of $\Ga(s)$ does not coincide with the root of $\Ga(u)$). Inspection of the map $\chi_u$ then shows that the left/right-rooted bipartite tree $\chi_u(\Ga(s))$ coincides with (a copy of)
${}^L\De(s)$ or $\De(s)^R$ depending on whether the root of $\Ga(s)$ is a left or a right vertex. Similarly, inspection of the map $\chi_{u(s\to t)}$ shows that $\chi_{u(s\to t)}(\Ga(t))$ (viewed as a left/right-rooted bipartite
tree depending on whether the root of $\Ga(t)$ in $\Ga(u(s\to t))$ is a left/right vertex) coincides with ${}^L\De(t)$ respectively $\De(t)^R$. Altogether, this means that $\De(u(s\to t))$ can be obtained from $\De(u)$ by replacing $\chi_u(\Ga(s))={}^L\De(s)$ by ${}^L\De(t)$ respectively $\chi_u(\Ga(s))=\De(s)^R$  by $\De(t)^R$.

As an example, consider  $s=(v\wedge w)\wedge y$, $t=(v\wedge (x\wedge w))\wedge (y\wedge x)$ and again $u=((x\wedge(v\wedge z))\wedge x)\wedge((v\wedge z)\wedge(\underline{(v\wedge w)\wedge y}))$ with $s$ the underlined subword. Note that the root of $\Ga(s)$ inside $\Ga(u)$ is a right vertex. Hence, in $\De(u)$ we have to replace the right-rooted tree $\chi_u(\Ga(s))=\De(s)^R=$
\tikz[baseline=1pt,scale=0.3]{\coordinate [label=left:$v$] (1) at (0,0.5);\coordinate [label=right:$y$] (2) at (2,1);
\coordinate[label=right:$w$] (3) at (2,0);  \draw (2,1) -- (0,0.5) --(2,0); \draw (0,0.5) node{$\bullet$}; \draw (2,1) node{$\bullet$}; \draw (2,0) node{$\bullet$};\draw (2)circle (10pt); } by $\De(t)^R$. Inside the graph $\De(u)$, $\chi_u(\Ga(s))$ is the indicated subtree:

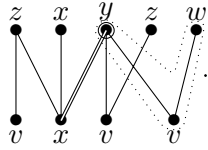
\centerline{
\tikz[scale=0.6]{
\coordinate[label=below:$v$] (1) at (0,0);
\coordinate[label=below:$x$] (2) at (1,0);
\coordinate[label=below:$v$] (3) at (2,0);
\coordinate[label=below:$v$] (4) at (3.5,0);
\coordinate[label=above:$z$] (5) at (0,2);
\coordinate[label=above:$x$] (6) at (1,2);
\coordinate[label=above:$z$] (7) at (3,2);
\coordinate[label=above:$y$] (8) at (2,2);
\coordinate[label=above:$w$] (9) at (4,2);
\draw (1) node {$\bullet$};
\draw (2) node {$\bullet$};
\draw (3) node {$\bullet$};
\draw (4) node {$\bullet$};
\draw (5) node {$\bullet$};
\draw (6) node {$\bullet$};
\draw (7) node {$\bullet$};
\draw (8) node {$\bullet$};
\draw (9) node {$\bullet$};
\draw (1)--(5)--(2)--(6);
\draw (7)--(3)--(8)--(4)--(9);
\draw[double](2)--(8);
\draw[dotted,rounded corners=3pt] (2,2.3)--(3.5,1)--(3.8,2.2)--(4.2,2.2)--(3.7,-0.2)--(3.3,-0.2)--(1.7,2)--cycle;
\draw (8)circle (5pt);
\filldraw (4.2,1)circle (0.5pt);}
}

{\parindent=0pt One readily checks that }

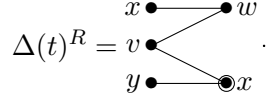
\centerline{\tikz[scale=0.5]{
\draw (-2.3,0) node{$\De(t)^R=$};
\coordinate[label=left:$y$] (1) at (0,-1);
\coordinate[label=left:$v$] (2) at (0,0);
\coordinate[label=left:$x$] (3) at (0,1);
\coordinate[label=right:$x$] (4) at (2,-1);
\coordinate[label=right:$w$] (5) at (2,1);
\draw (1) node {$\bullet$};
\draw (2) node {$\bullet$};
\draw (3) node {$\bullet$};
\draw (4) node {$\bullet$};
\draw (5) node {$\bullet$};
\draw (4)circle (6pt);
\draw (1)--(4)--(2)--(5)--(3);
\filldraw (3,0)circle (0.5pt);
}}

{\parindent=0pt Hence, $\chi_u(\Ga(s))$ replaced with $\De(t)^R$ yields}

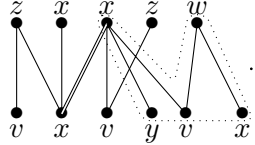
\centerline{
\tikz[scale=0.6]{
\coordinate[label=below:$v$] (1) at (0,0);
\coordinate[label=below:$x$] (2) at (1,0);
\coordinate[label=below:$v$] (3) at (2,0);
\coordinate[label=below:$v$] (4) at (3.75,0);
\coordinate[label=above:$z$] (5) at (0,2);
\coordinate[label=above:$x$] (6) at (1,2);
\coordinate[label=above:$z$] (7) at (3,2);
\coordinate[label=above:$x$] (8) at (2,2);
\coordinate[label=above:$w$] (9) at (4,2);
\coordinate[label=below:$y$] (10) at (3,0);
\coordinate[label=below:$x$] (11) at (5,0);
\draw (1) node {$\bullet$};
\draw (2) node {$\bullet$};
\draw (3) node {$\bullet$};
\draw (4) node {$\bullet$};
\draw (5) node {$\bullet$};
\draw (6) node {$\bullet$};
\draw (7) node {$\bullet$};
\draw (8) node {$\bullet$};
\draw (9) node {$\bullet$};
\draw (10) node {$\bullet$};
\draw (11) node {$\bullet$};
\draw (1)--(5)--(2)--(6);
\draw (7)--(3)--(8)--(4)--(9);
\draw[double](2)--(8);
\draw (9)--(11); \draw (8)--(10);
\draw[dotted,rounded corners=3pt] (2,2.25)--(3.5,0.7)--(3.8,2.2)--(4.2,2.2)--(5.25,-0.15)--(2.8,-0.15)--(1.75,2.1)--cycle;
\filldraw (5.2,1)circle (0.5pt);}
}

This procedure can be applied simultaneously to several pairwise disjoint subwords $s_1,s_2,\dots, s_k$. A particularly important instance of the latter is when each letter $x$ in $u$ is substituted with a certain word $w_x$.
In other words, for an endomorphism $\psi:F_2(X)\to F_2(X)$ we shall describe $\De(u\psi)$ in terms of $\De(u)$ and $\De(x_i\psi)$ ($i=1,\dots, n$) where $x_1,\dots, x_n$ are the letters occurring in $u$. Indeed, for a vertex $a$ of $\De(u)$ denote by $\cb_a$ the label of $a$. In order to apply the procedure described above (to all one-letter subwords  simultaneously) replace in $\De(u)$ each left vertex $a$ by the left-rooted tree $\Ls(a,\psi):={}^L\De(\cb_a\psi)$ and each right vertex $b$ by the right-rooted tree $\Rs(b,\psi):=\De(\cb_b\psi)^R$ (all these graphs shall be assumed to be pairwise disjoint). The result is the graph $\De(u\psi)$ where the distinguished left vertex of $\De(u\psi)$ is the distinguished (left) vertex of $\Ls(\lv_{\De(u)},\psi)$ while the distinguished right vertex of $\De(u\psi)$ is the distinguished (right) vertex of $\Rs(\rv_{\De(u)},\psi)$.
Alternatively, $\De(u\psi)$ can be obtained as follows: form the disjoint union
$$\left(\bigcup_{a\in L_{\De(u)}}\Ls(a,\psi)\right)\,\cup\,\left( \bigcup_{b\in R_{\De(u)}}\Rs(b,\psi)\right)$$
of all graphs $\Ls(a,\psi)$ and $\Rs(b,\psi)$ and add, for each edge $(a,b)$ of $\De(u)$ the edge $(\lv_{\Ls(a,\psi)},\rv_{\Rs(b,\psi)})$. For $(a,b)=(\lv_{\De(u)},\rv_{\De(u)})$ this yields the connection between the distinguished vertices $\lv_{\Ls(\lv_{\De(u)},\psi)}$ and $\rv_{\Rs(\rv_{\De(u)},\psi)}$ of $\De(u\psi)$.

There is a subtree of $\De(u\psi)$, the \emph{skeleton} $\sk(u,\psi)$, which is the subtree of $\De(u\psi)$ spanned by the set of vertices
$$\{\lv_{\Ls(a,\psi)}\mid a\in L_{\De(u)}\}\cup\{\rv_{\Rs(b,\psi)}\mid b\in R_{\De(u)}\}$$
or likewise, spanned by all edges $(\lv_{\Ls(a,\psi)},\rv_{\Rs(b,\psi)})$ for $(a,b)$ an edge in $\De(u)$.
The graph $\sk(u,\psi)$ has the same structure as $\De(u)$ except that the labels of the vertices have changed. Indeed, the label of each left vertex $a$ of $\De(u)$ is changed from $\cb_a$ to $\l(\cb_a\psi)$ and the label of each right vertex $b$ of $\De(u)$ is changed from $\cb_b$ to $\r(\cb_b\psi)$. In case the left content of $u$ is disjoint from its right content, the skeleton $\sk(u,\psi)$ itself can be viewed as a graph of the form $\De(u\psi')$ for any endomorphism $\psi'$ satisfying $x\psi'=\l(x\psi)$ if $x\in \co_l(u)$ and $x\psi'=\r(x\psi)$ if $x\in \co_r(u)$.

An immediate consequence of the description of $\De(u\psi)$ is that $\B(X)$ is a relatively free binary algebra, that is, the kernel $\mathrm{ker}\De$ of the homomorphism $\De:F_2(X)\to \B(X)$ is a fully invariant congruence.

\begin{Cor} \label{deltafullyinvariant} For all $u,v\in F_2(X)$ and each endomorphism $\psi:F_2(X)\to F_2(X)$, if $\De(u)=\De(v)$ then $\De(u\psi)=\De(v\psi)$. In particular, $\mathrm{ker}\De$ is a fully invariant congruence on $F_2(X)$.
\end{Cor}

\section{The free pseudosemilattice}\label{freePS}
\subsection{The binary algebra $\A(X)$}
Let $\al\in \B'(X)$; a degree $1$ vertex $a$ of $\al$ together with the unique edge $e$ that has $a$ as one of its endpoints is a \emph{thorn} if the two vertices connected by $e$ have the same label. A thorn is \emph{essential} if its vertex is a distinguished one, otherwise it is \emph{non-essential}. We introduce two reduction rules for modifying a member $\al$ of $\B'(X)$:

\begin{enumerate}
\item[(i)] remove a non-essential thorn $\{e,a\}$ from $\al$. This rule may be visualized graphically as

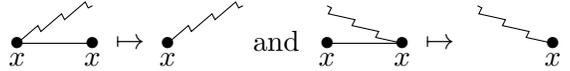
\centerline{\tikz[baseline=-0.1cm,scale=0.5]{
\coordinate[label=below:$x$] (1) at (0,0);
\coordinate[label=below:$x$] (2) at (2,0);
\coordinate[label=below:$x$] (3) at (4,0);
\draw (1) node{$\bullet$};
\draw (2) node{$\bullet$};
\draw (3) node{$\bullet$};
\draw (3,0) node{$\mapsto$};
\draw (1)--(2);
\draw[decorate,decoration=saw] (1)--(2,1);
\draw[decorate,decoration=saw] (3)--(6,1);
}
and \tikz[baseline=-0.1cm,scale=0.5]{
\coordinate[label=below:$x$] (1) at (0,0);
\coordinate[label=below:$x$] (2) at (2,0);
\coordinate[label=below:$x$] (3) at (6,0);
\draw (1) node{$\bullet$};
\draw (2) node{$\bullet$};
\draw (3) node{$\bullet$};
\draw (3,0) node{$\mapsto$};
\draw (1)--(2);
\draw[decorate,decoration=saw] (2)--(0,1);
\draw[decorate,decoration=saw] (3)--(4,1);
\filldraw (6.5,0)circle (0.5pt);}}
\item[(ii)] suppose that two edges $e$ and $f$ have a vertex in common and that the two other (distinct) vertices $a$ and $b$ have the same label; then identify the two edges $e$ and $f$ and the vertices $a$ and $b$ (and retain their label).
    If one of the merged vertices happens to be a distinguished one then so is
the resulting vertex. Graphically, this rule may be visualized as

    
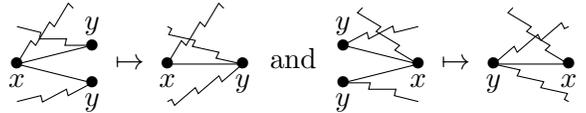
\centerline{
\tikz[baseline=0.2cm,scale=0.5]{\coordinate [label=below:$x$] (1) at (0,0.5);\coordinate [label=above:$y$] (2) at (2,1);
\coordinate[label=below:$y$] (3) at (2,0); \coordinate [label=below:$x$] (4) at (4,0.5); \coordinate [label=below:$y$] (5) at (6,0.5); \draw (2,1) -- (0,0.5) --(2,0); \draw (0,0.5) node{$\bullet$}; \draw (2,1) node{$\bullet$}; \draw (2,0) node{$\bullet$};\draw (3,0.5) node{$\mapsto$};\draw (4,0.5)--(6,0.5);\draw (4,0.5) node{$\bullet$}; \draw[decorate,decoration=saw](0,0.5)--(1.5,2);
\draw[decorate,decoration=saw](2,0)--(0,-0.5);
\draw[decorate,decoration=saw](2,1)--(0,1.5);
\draw[decorate,decoration=saw](4,0.5)--(5.5,2);
\draw[decorate,decoration=saw](6,0.5)--(4,-0.5);
\draw[decorate,decoration=saw](6,0.5)--(4,1.5);
\draw(6,0.5)node{$\bullet$};}
and \tikz[baseline=0.2cm,scale=0.5]{
\coordinate[label=above:$y$] (1) at (0,1);
\coordinate[label=below:$y$] (2) at (0,0);
\coordinate[label=below:$x$] (3) at (2,0.5);
\coordinate[label=below:$y$] (4) at (4,0.5);
\coordinate[label=below:$x$] (5) at (6,0.5);
\draw (1) node{$\bullet$};
\draw (2) node{$\bullet$};
\draw (3) node{$\bullet$};
\draw (4) node{$\bullet$};
\draw (5) node{$\bullet$};
\draw (3,0.5) node{$\mapsto$};
\draw (1)--(3)--(2); \draw (4)--(5);
\draw[decorate,decoration=saw](0,1)--(2,1.5);
\draw[decorate,decoration=saw](0,0)--(2,-0.5);
\draw[decorate,decoration=saw](2,0.5)--(0.5,2);
\draw[decorate,decoration=saw](4,0.5)--(6,1.5);
\draw[decorate,decoration=saw](4,0.5)--(6,-0.5);
\draw[decorate,decoration=saw](6,0.5)--(4.5,2);
\filldraw (6.5,0.5)circle (0.5pt);}
}
\end{enumerate}

Rule (i) is referred to as the \emph{deletion of a thorn} while rule (ii) is called an \emph{edge-folding}. If we apply the reductions (i) and (ii) to $\al\in \B'(X)$ in any order until no more reduction is possible then we obtain the \emph{reduced form} $\ol{\al}$ of $\al$. Note that $\ol\al$ is uniquely determined and does not depend on the order the reductions are applied. Let $\A(X):=\ol{\B'(X)}$ be the set of all reduced members of $\B'(X)$. Setting $\ol\al:=$\tikz[baseline=2.25pt,scale=0.3]{\coordinate [label=below:$x$] (4) at (0,0.5); \coordinate [label=below:$x$] (5) at (2.5,0.5);\draw[double] (4) --(5);\draw (4) node{$\bullet$};\draw (5) node{$\bullet$};} for $\al=\underset{x}{\bullet}$, the mapping $\al\to\ol\al$ is  surjective from $\B(X)$ onto $\A(X)$. We define a binary operation $\wedge$ on $\A(X)$ by the rule
$$\al\wedge\be:=\ol{\al\sqcap \be}.$$
For $\al,\be\in \B'(X)$ we have $\ol{\ol{\al}\sqcap\ol{\be}}=\ol{\al\sqcap\be}$ since the reductions to obtain $\ol{\al\sqcap\be}$ from $\al\sqcap\be$ may be applied in any order. This statement stays true if $\al$ and/or $\be$ is taken from $\B(X)=\B'(X)\cup X$. It follows that the mapping $\al\mapsto \ol\al$ is a surjective homomorphism $\B(X)\to \A(X)$.  Note that $\A(X)$ is an $X$-generated binary algebra if we identify $x\in X$ with \tikz[baseline=2.25pt,scale=0.3]{\coordinate [label=below:$x$] (4) at (0,0.5); \coordinate [label=below:$x$] (5) at (2.5,0.5);\draw[double] (4) --(5);\draw (4) node{$\bullet$};\draw (5) node{$\bullet$};}. We denote the canonical homomorphism $F_2(X)\to \A(X)$ by $\Te$. Subject to this notation, $\Te(u)=\ol{\De(u)}$ for each $u\in F_2(X)$.

Next we show that the kernel of $\Te$ is also a fully invariant congruence, that is, $(\A(X),\wedge)$ is also a relatively free binary algebra generated by $$\{\text{\tikz[baseline=2.25pt,scale=0.3]{\coordinate [label=below:$x$] (4) at (0,0.5); \coordinate [label=below:$x$] (5) at (2.5,0.5);\draw[double] (4) --(5);\draw (4) node{$\bullet$};\draw (5) node{$\bullet$};}}\mid x\in X\}.$$
We start by proving some auxiliary facts about the binary operation $\wedge$ in $\A(X)$.
\begin{Lemma}\label{idempotency} For each $\ga\in \A(X)$ we have $\ga\wedge\ga=\ga$.
\end{Lemma}
\begin{proof} We need to show that $\ol{\ga\sqcap\ga}=\ga$. Let $\ga'$ be a disjoint copy of $\ga$. We form the graph
$${}^L\ga'\cup\{(\lv_{\ga'},\rv_\ga)\}\cup \ga^R$$
and observe that the two edges $(\lv_{\ga'},\rv_\ga)$ and $(\lv_\ga,\rv_\ga)$ may be identified by a first edge-folding. In particular, the  respective distinguished left vertices $\lv_{\ga'}$ and $\lv_\ga$ will be identified. Now let $a$ be an arbitrary vertex of $\ga$ and let
$$\lv_\ga, a_1,\dots,a_n,a$$
be the geodesic path in $\ga$ starting at $\lv_\ga$ and ending at $a$. Let
$$\lv_\ga=\lv_{\ga'},a_1',\dots, a_n',a'$$
be the corresponding path in $\ga'$. Since for each $i$, $a_i$ and $a_i'$ have the same label we may  identify the edges $(\lv_\ga,a_1)$ and $(\lv_\ga,a_1')$ and thus the vertices $a_1$ and $a_1'$ will be identified. Next we may identify the edges $(a_2,a_1)$ and $(a_2',a_1)$ and thus $a_2$ and $a_2'$ will be identified. Eventually the vertex $a$ will be identified with the vertex $a'$. In that process, each edge $(a,b)$ of $\ga$ will be identified with its counterpart $(a',b')$ in $\ga'$. After a finite sequence of edge-foldings we arrive at a graph isomorphic with $\ga$ which is reduced by assumption. Altogether, $\ol{\ga\sqcap\ga}=\ga$.
\end{proof}
\begin{Lemma}\label{foldingsubstitution} Let $\ga\in\B'(X)$ and let $\ga'$ be a disjoint copy of $\ga$. Let $\al$ be the bipartite tree defined by
$$\al=\ga\cup\{a,(a,\rv_\ga),(a,\rv_{\ga'})\}\cup \ga'$$
where $a$ is a new, arbitrarily labeled left vertex adjoined to $\ga\cup\ga'$. Then $\al$ can be reduced by a sequence of edge-foldings   to a graph isomorphic to $\{a,(a,\rv_\ga)\}\cup \ga$. The analogous assertion for $\al=\ga\cup\{b,(\lv_\ga,b),(\lv_{\ga'},b)\}\cup \ga'$ with $b$ a new right vertex also holds.
\end{Lemma}
\begin{proof} A first edge-folding may be used to identify the edges $(a,\rv_\ga)$ and $(a,\rv_{\ga'})$ and thus the vertices $\rv_\ga$ and $\rv_{\ga'}$ are also identified. Now we may proceed as in Lemma \ref{idempotency} and show that each vertex and edge of $\ga$ is identified with its counterpart in $\ga'$.
\end{proof}
\begin{Cor}\label{thetafullyinvariant} Let $\psi$ be an endomorphism of $F_2(X)$ and let $u,u'\in F_2(X)$ be such that $\De(u')$ is obtained from $\De(u)$ by a single deletion of a thorn or a single edge-folding; then $\Te(u\psi)=\Te(u'\psi)$.
\end{Cor}
\begin{proof} (i) Suppose that $\De(u')$ is obtained by deletion of a thorn $\{e,b\}$ in $\De(u)$ with edge $e=(a,b)$ (the dual case is analogous). According to the description of $\De(v\psi)$ for arbitrary $v\in F_2(X)$, $\De(u\psi)$ has a subgraph consisting of the disjoint union $\Ls(a,\psi)\cup\Rs(b,\psi)$ together with the edge connecting the distinguished left vertex of $\Ls(a,\psi)$ and the distinguished right vertex of $\Rs(b,\psi)$. Recall that  $\Ls(a,\psi)={}^L\De(x\psi)$ while $\Rs(b,\psi)=\De(x\psi)^R$ for
 $x=\cb_a=\cb_b$.   The proof of Lemma \ref{idempotency} shows that a sequence of edge-foldings reduces this subgraph of $\De(u\psi)$ to $\Ls(a,\psi)(={}^L\De(x\psi))$. The entire graph $\De(u\psi)$ has thus been reduced to $\De(u'\psi)$. Consequently, $\ol{\De(u\psi)}=\ol{\De(u'\psi)}$.

(ii) Suppose that $\De(u')$ is obtained from $\De(u)$ by an edge-folding of the form
$$
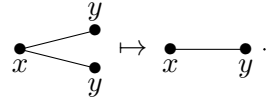
\tikz[scale=0.5]{\coordinate [label=below:$x$] (1) at (0,0.5);\coordinate [label=above:$y$] (2) at (2,1);
\coordinate[label=below:$y$] (3) at (2,0); \coordinate [label=below:$x$] (4) at (4,0.5); \coordinate [label=below:$y$] (5) at (6,0.5); \draw (2,1) -- (0,0.5) --(2,0); \draw (0,0.5) node{$\bullet$}; \draw (2,1) node{$\bullet$}; \draw (2,0) node{$\bullet$};\draw (3,0.5) node{$\mapsto$};\draw (4,0.5)--(6,0.5);\draw (4,0.5) node{$\bullet$}; \draw(6,0.5)node{$\bullet$};\filldraw (6.5,0.5)circle (0.5pt);}$$
Denote the left vertex of \tikz[baseline=0pt,scale=0.3]{\coordinate [label=left:$x$] (1) at (0,0.5);\coordinate [label=right:$y$] (2) at (2,1);
\coordinate[label=right:$y$] (3) at (2,0);  \draw (2,1) -- (0,0.5) --(2,0); \draw (0,0.5) node{$\bullet$}; \draw (2,1) node{$\bullet$}; \draw (2,0) node{$\bullet$};} by $a$ and the two right vertices by $b$ and $c$, respectively.
Consider the subgraph of $\De(u\psi)$ obtained by replacing in \tikz[baseline=0pt,scale=0.3]{\coordinate [label=left:$x$] (1) at (0,0.5);\coordinate [label=right:$y$] (2) at (2,1);
\coordinate[label=right:$y$] (3) at (2,0);  \draw (2,1) -- (0,0.5) --(2,0); \draw (0,0.5) node{$\bullet$}; \draw (2,1) node{$\bullet$}; \draw (2,0) node{$\bullet$};}
 the left vertex with $\Ls(a,\psi)(={}^L\De(x\psi))$ and the two right vertices $b$ and $c$ with $\Rs(b,\psi)$ and $\Rs(c,\psi)$ both of which are (disjoint) copies of  $\De(y\psi)^R$. Lemma \ref{foldingsubstitution} then shows that through a sequence of edge foldings we get the graph obtained from \tikz[baseline=0pt,scale=0.3]{\coordinate [label=below:$x$] (4) at (0,0.5); \coordinate [label=below:$y$] (5) at (2.5,0.5);\draw (4) --(5);\draw (4) node{$\bullet$};\draw (5) node{$\bullet$};} by substitution of the left vertex with $\Ls(a,\psi)={}^L\De(x\psi)$ and the right vertex with $\Rs(b,\psi)=\De(y\psi)^R$. But the latter is a subgraph of $\De(u'\psi)$. Altogether we have reduced by a sequence of edge-foldings $\De(u\psi)$ to $\De(u'\psi)$. It follows that $\ol{\De(u\psi)}=\ol{\De(u'\psi)}$.
\end{proof}
\begin{Cor} The binary algebra $(\A(X),\wedge)$ is a relatively free pseudosemilattice on $X$.
\end{Cor}
\begin{proof} Let $\psi:F_2(X)\to F_2(X)$ be an endomorphism and let $u,v\in F_2(X)$ be such that $\Te(u)=\Te(v)$. Then there exist $u_0,\dots,u_n,v_0,\dots,v_m\in F_2(X)$ such that
\begin{itemize}
\item $\De(u)=\De(u_0)$,
\item for each $i=0,\dots, n-1$ the graph $\De(u_{i+1})$ is obtained from $\De(u_i)$ by the deletion of a thorn or an edge-folding,
\item $\De(u_n)=\ol{\De(u)}=\Te(u)=\Te(v)=\ol{\De(v)}=\De(v_m)$,
\item for each $j=m,\dots,1$ the graph $\De(v_j)$ is obtained from $\De(v_{j-1})$ by the deletion of a thorn or an edge folding,
\item $\De(v_0)=\De(v)$.
\end{itemize}
Lemma \ref{deltafullyinvariant} and Corollary \ref{thetafullyinvariant} then imply
$$\Te(u\psi)=\ol{\De(u_0\psi)}=\dots=\ol{\De(u_n\psi)}=\ol{\De(v_m\psi)}=\dots=\ol{\De(v_0\psi)}= \Te(v\psi).$$ It follows that $(\A(X),\wedge)$ is a relatively free binary algebra on $X$.

In order to show that $(\A(X),\wedge)$ is a pseudosemilattice it suffices now to show that the defining identities ((PS1)--(PS3), (PS2'), (PS3'))
for pseudosemilattices are \textbf{relations} satisfied by the free generators of $\A(X)$. However, the latter can be verified by a straightforward check.
\end{proof}

\subsection{The word problem for free pseudosemilattices} On each pseudosemilattice, two equivalence relations $\Rc$ and $\Lc$ are defined by $$e\Rc f\Leftrightarrow (e]_{\Rc}=(f]_{\Rc}\mbox{ and }e\Lc f\Leftrightarrow (e]_{\Lc}=(f]_{\Lc}$$ (these are the equivalence relations induced by the quasiorders ${\le}_{\Rc}$ and ${\le}_{\Lc}$). From the definition of the operation $\wedge$ given in the introduction it follows that for arbitrary elements $e,e',f,f'$ of any pseudosemilattice the implication
$$e\Rc e'\ \&\ f\Lc f'\Rightarrow e\wedge f=e'\wedge f'$$
holds.

Next we  derive some identities satisfied in each pseudosemilattice. These results are known from the second author's paper \cite{l1}. Here we present a proof that does not depend on locally inverse semigroups so that this part is self-contained. We use some notation from \cite{l1}: for $u_1,\dots, u_n\in F_2(X)$ set
$$(\wedge u_1\dots u_n):= (\dots((u_1\wedge u_2)\wedge u_3)\wedge\dots )\wedge u_n$$
and
$$(u_n\dots u_1\wedge):=u_n\wedge(\dots\wedge(u_3\wedge(u_2\wedge u_1))\dots).$$
From the identities (PS1), (PS2) and (PS2') the identities
\begin{equation}\label{basicid}
(x\wedge x)\wedge(x\wedge y)\approx x\wedge y\approx (x\wedge y)\wedge (y\wedge y)
\end{equation}
are easily derived. Moreover, identity (PS3) implies that in each pseudosemilattice $E$, for each $x\in E$, all elements of the form $x\wedge y$ ($y\in E$) generate a semigroup. It follows that each expression of the form
$$(x\wedge y_1)\wedge(x\wedge y_2)\wedge \dots \wedge (x\wedge y_n),$$
without setting any further brackets, defines a uniquely determined element in any pseudosemilattice $E$, once the variables $x,y_1\dots,y_n$ are substituted with elements of $E$. It is therefore justified to use such expressions as terms when dealing with identities of pseudosemilattices. By identity (PS3'), the same holds for expressions of the form
$$(y_1\wedge x)\wedge (y_2\wedge x)\wedge\dots\wedge(y_n\wedge x).$$
\begin{Lemma}\label{idforPS} For each $n\in \mathbb{N}$, each pseudosemilattice satisfies the identity
$$(\wedge xy_1\dots y_n)\approx (x\wedge y_1)\wedge(x\wedge y_2)\wedge \dots\wedge(x\wedge y_n)$$ and its dual.
\end{Lemma}
\begin{proof}
The claim is proved by  induction on $n$. There is nothing to prove for $n=1$ and for $n=2$ this is just the identity (PS2). Suppose that the claim be true for $n\ge 2$. Then
\begin{align*}
 (\wedge xy_1\dots y_{n+1})&=(\wedge xy_1\dots y_n)\wedge y_{n+1} \\
 &\approx [(x\wedge y_1)\wedge\dots\wedge(x\wedge y_n)]\wedge y_{n+1} \\
 &\approx [(x\wedge x)\wedge\{(x\wedge y_1)\wedge\dots\wedge(x\wedge y_n)\}]\wedge y_{n+1}\\
  &\approx [(x\wedge x)\wedge(x\wedge y_1)\wedge\dots\wedge(x\wedge y_n)]\wedge [(x\wedge x)\wedge y_{n+1}] \\
&\approx (x\wedge y_1)\wedge\dots\wedge(x\wedge y_n)\wedge (x\wedge y_{n+1})\, .
 \end{align*}
The dual identity is proved by symmetry.
 \end{proof}

 The next statement is an immediate consequence.
 \begin{Cor}\label{idforfolding} For all $n,k\in \mathbb{N}$, each pseudosemilattice satisfies the identities
 \begin{enumerate}
 \item $(\wedge xy_1\dots y_n)\wedge(\wedge xz_1\dots z_k)\approx (\wedge xy_1\dots y_nz_1\dots z_k)$,
 \item $(\wedge xy_1\dots y_nxz)\approx (\wedge xy_1\dots y_nz)$
 \end{enumerate}
 and their duals.
 \end{Cor}
 A final auxiliary result that is crucial is the following.

 \begin{Lemma}\label{idforR} Each pseudosemilattice satisfies $(\wedge xyzw)\approx (\wedge xzyw)$ and $(\wedge xyzy)\approx (\wedge xzy)$ and their duals.
 \end{Lemma}
 \begin{proof} We have
 \begin{align*}
 (\wedge xyzw)&\approx [(x\wedge y)\wedge(x\wedge z)]\wedge (x\wedge w) &\mbox{ by Lemma \ref{idforPS}}\\
 &\approx (x\wedge y)\wedge (x\wedge z)\wedge (x\wedge y) \wedge (x\wedge w) &\mbox{ by (PS2)}\\
 &\approx \{(x\wedge \underline y)\wedge[(x\wedge z)\wedge \underline y]\}\wedge (x\wedge w) &\mbox{ by (PS2)}\\
 &\approx \{x\wedge [(x\wedge z)\wedge y]\}\wedge (x\wedge w) &\mbox{ by (PS2')}\\
 &\approx \{x\wedge [(x\wedge z)\wedge (x\wedge y)]\}\wedge (x\wedge w) &\mbox{ by (PS2)}\\
 &\approx (x\wedge x)\wedge (x\wedge z)\wedge (x\wedge y)\wedge (x\wedge w) &\mbox{ by (PS1)}\\
 &\approx (x\wedge z)\wedge (x\wedge y)\wedge (x\wedge w) &\mbox{ by (\ref{basicid})}\\
 &\approx (\wedge xzyw). &\mbox{ by Lemma \ref{idforPS}}
 \end{align*}
In particular $(\wedge xyzy)\approx (\wedge xzyy)\approx (\wedge xzy)$ where the last identity follows from Lemma \ref{idforPS} and (PS1). The duals are again proved by symmetry.
 \end{proof}

Lemma \ref{idforR} implies that for arbitrary elements $x,y,z$ of a pseudosemilattice $E$, setting $e=(x\wedge y)\wedge z$ and $f=(x\wedge z)\wedge y$ the equalities $e\wedge f=f$ and $f\wedge e=e$ hold, so that $(x\wedge y)\wedge z\Rc (x\wedge z)\wedge y$. By symmetry we also get $y\wedge (z\wedge x)\Lc z\wedge(y\wedge x).$

\begin{Cor}\label{RLinPS} $$\mathbf{PS}\models (x\wedge y)\wedge z\Rc (x\wedge z)\wedge y\mbox{ and }
\mathbf{PS}\models y\wedge(z\wedge x)\Lc z\wedge (y\wedge x).$$
\end{Cor}
We are going to prove that for all $u,v\in F_2(X)$, $\Te(u)=\Te(v)$ implies $\mathbf{PS}\models u\approx v$. This immediately shows that $(\A(X),\wedge)$ is a model of the free $X$-generated pseudosemilattice. The proof is essentially divided into two parts the first of which is formulated as follows.
\begin{Prop}\label{DeaboveTe} For all $u,v\in F_2(X)$ the following implications hold.
\begin{enumerate}
\item ${}^L\De(u)={}^L\De(v) \Rightarrow \mathbf{PS}\models u\Rc v,$
\item $\De(u)^R= \De(v)^R \Rightarrow \mathbf{PS}\models u\Lc v,$
\item $\De(u)=\De(v) \Rightarrow \mathbf{PS}\models u\approx v.$
\end{enumerate}
\end{Prop}
\begin{proof} The proof is by simultaneous induction on the number $\vert V(\De(u))\vert$ of vertices of $\De(u)$.
All three implications are true (by Corollary \ref{RLinPS}) for any $u,v$ with $|V(\De(u))|=|V(\De(v))|\le 3$. So, let $u,v\in F_2(X)$ with $|V(\De(u))|=|V(\De(v))|\ge 4$ and assume first that ${}^L\De(u)={}^L\De(v)$ but $\De(u)\ne\De(v)$. Let $u=u_1\wedge u_2$ and $v=v_1\wedge v_2$. By definition,
$$\De(u)={}^L\De(u_1)\cup\{(\lv,\rv)\}\cup \De(u_2)^R$$
and
$$\De(v)={}^L\De(v_1)\cup\{(\lv,\sv)\}\cup \De(v_2)^R$$
where $\lv=\lv_{\De(u)}=\lv_{\De(v)}$ and $\rv=\rv_{\De(u)}\ne\rv_{\De(v)}=\sv$. If we delete the two edges $(\lv,\rv)$ and $(\lv,\sv)$ from the graph then it is decomposed into three pairwise disjoint trees $T_l,T_r,T_s$ (namely the connected components containing $\lv,\rv,\sv$, respectively). We consider $T_l$ as left-rooted tree with left root $\lv$ and $T_r$ and $T_s$ as right-rooted trees
with right roots $\rv$ and $\sv$, respectively. Then we have
$${}^L\De(u_1)=T_l\cup\{(\lv,\sv)\}\cup T_s\mbox{ and }\De(u_2)^R=T_r$$
while
$${}^L\De(v_1)=T_l\cup\{(\lv,\rv)\}\cup T_r\mbox{ and }\De(v_2)^R=T_s$$
(the distinguished left root in both cases being $\lv$). Let now $w_l,w_r,w_s\in F_2(X)$ be such that
$${}^L\De(w_l)=T_l,\ \De(w_r)^R=T_r,\ \De(w_s)^R=T_s.$$ It follows from the construction that
$${}^L\De(u_1)={}^L\De(w_l\wedge w_s)\mbox{ and } {}^L\De(v_1)={}^L\De(w_l\wedge w_r).$$
By the induction hypothesis this implies that
$$\mathbf{PS}\models u_1\Rc w_l\wedge w_s,\  v_1\Rc w_l\wedge w_r.$$
Once more by construction we have
$$\De(u_2)^R=\De(w_r)^R\mbox{ and }\De(v_2)^R=\De(w_s)^R$$
which, again by the induction hypothesis implies that
$$\mathbf{PS}\models u_2\Lc w_r,\ v_2\Lc w_s.$$
Altogether,
$$\mathbf{PS}\models u_1\wedge u_2\approx (w_l\wedge w_s)\wedge w_r,\ v_1\wedge v_2\approx (w_l\wedge w_r)\wedge w_s.$$
From Corollary \ref{RLinPS} it follows that $\mathbf{PS}\models u_1\wedge u_2\Rc v_1\wedge v_2$, as required.

The implication
$$\De(u)^R=\De(v)^R\ \& \ \De(u)\ne \De(v)\Rightarrow \mathbf{PS}\models u\Lc v$$
is proved by symmetry.

Finally, assume that $\De(u)=\De(v)$ for $u=u_1\wedge u_2$ and $v=v_1\wedge v_2$. Then ${}^L\De(u_1)={}^L\De(v_1)$ and $\De(u_2)^R=\De(v_2)^R$. By induction hypothesis we have $\mathbf{PS}\models u_1\Rc v_1,\ u_2\Lc v_2$ which implies $\mathbf{PS}\models u_1\wedge u_2\approx v_1\wedge v_2$, as required.
\end{proof}
The second part consists of proving the following claim.
\begin{Prop} Let $u,u'\in F_2(X)$ be such that $\De(u')$ is obtained from $\De(u)$ by a single deletion of a thorn or a single edge-folding; then $\mathbf{PS}\models u\approx u'$.
\end{Prop}
\begin{proof} For both cases we use induction on $|V(\De(u))|$ and we may assume that $|c(u)|\ge 2$.
Let us first assume that the reduction $\De(u)\to \De(u')$ is by deletion of a thorn. The induction base is easily checked directly by looking at the graphs on three vertices. Let $u\in F_2(X)$ be such that $\vert V(\De(u))\vert\ge 4$ and suppose that the claim be true for all $v$ with $|V(\De(v))|<|V(\De(u))|$. Now $u=u_1\wedge u_2$ and $\De(u)={}^L\De(u_1)\cup\{(\lv,\rv)\}\cup \De(u_2)^R$. Let $\{e,a\}$ be the thorn to be deleted to obtain $\De(u')$ from $\De(u)$. The edge $e$ cannot coincide with $(\lv,\rv)$ so it must belong to ${}^L\De(u_1)$ or $\De(u_2)^R$. Suppose the former is true (the latter case is handled completely analogously). We shall distinguish two cases: $\rv_{\De(u_1)}\ne a$ and $\rv_{\De(u_1)}= a$. In the first case, the deletion of $\{e,a\}$ does not remove the right distinguished vertex of $\De(u_1)$.  Let $\ga=\De(u_1)\setminus\{e,a\}$ be the  graph obtained from $\De(u_1)$ by deleting the thorn $\{e,a\}$ and let $u_1'\in F_2(X)$ be chosen with $\De(u_1')=\ga$. Then by induction hypothesis $\mathbf{PS}\models u_1\approx u_1'$, and therefore also $\mathbf{PS}\models u_1\wedge u_2\approx u_1'\wedge u_2.$ Now let $u'$ be any word such that $\De(u')=\De(u)\setminus\{e,a\}=\De(u_1'\wedge u_2)$. Then
 by Proposition \ref{DeaboveTe},
$\mathbf{PS}\models u_1'\wedge u_2\approx u'$ and therefore also $\mathbf{PS}\models u\approx u'$. In the second case, $u_1=u_1'\wedge x$ and $x=\l(u)=\l(u_1)$. Here it follows from Corollary \ref{idforfolding} that $\mathbf{PS}\models u=(u_1'\wedge x)\wedge u_2\approx u_1'\wedge u_2$. Finally, if $u'$ is any word with $\De(u')=\De(u)\setminus \{e,a\}$ then again by  Proposition \ref{DeaboveTe}
$\mathbf{PS}\models u_1'\wedge u_2\approx u'$ whence also $\mathbf{PS}\models u\approx u'$.

Let us now consider the case when the reduction $\De(u)\to \De(u')$ is by an edge-folding. The induction base is again directly checked by inspection of the graphs on three vertices. Let $u\in F_2(X)$ with $|V(\De(u))|\ge 4$ and suppose the claim be true for all $v\in F_2(X)$ for which $|V(\De(v))|<|V(\De(u))|$ and let $u=u_1\wedge u_2$. Assume the edge-folding is of the form (the dual case is treated analogously):
\begin{equation}\label{folding}
\tikz[baseline=7.5pt,scale=0.5]{\coordinate [label=below:$x$] (1) at (0,0.5);\coordinate [label=above:$y$] (2) at (2,1);
\coordinate[label=below:$y$] (3) at (2,0); \coordinate [label=below:$x$] (4) at (4,0.5); \coordinate [label=below:$y$] (5) at (6,0.5); \draw (2,1) -- (0,0.5) --(2,0); \draw (0,0.5) node{$\bullet$}; \draw (2,1) node{$\bullet$}; \draw (2,0) node{$\bullet$};\draw (3,0.5) node{$\mapsto$};\draw (4,0.5)--(6,0.5);\draw (4,0.5) node{$\bullet$}; \draw(6,0.5)node{$\bullet$};\filldraw (6.5,0.5)circle (0.5pt);}\end{equation}
Let $a$ denote the involved left vertex and $b,c$ the two involved right vertices. If neither of the pair $(a,b)$ and $(a,c)$ is the pair of distinguished vertices then both edges are contained  either in $\De(u_1)$ or $\De(u_2)$ and we may assume the former. Then the folding (\ref{folding}) reduces  $\De(u_1)$ to a graph $\gamma$. Let $u_1'\in F_2(X)$ with $\De(u_1')=\ga$. By induction hypothesis, $\mathbf{PS}\models u_1\approx u_1'$ and hence $\mathbf{PS}\models u_1\wedge u_2\approx u_1'\wedge u_2$. If $u'$ is any word such that $\De(u')$ is obtained from $\De(u)$ by the edge-folding mentioned above then $\De(u')=\De(u_1'\wedge u_2)$, whence $\mathbf{PS}\models u'\approx u_1'\wedge u_2$  and we are done. Hence we may assume that $(a,b)$ is the distinguished pair of vertices of $\De(u)$. Then the edge $(a,c)$ belongs to ${}^L\De(u_1)$ (and $a$ is the distinguished left vertex of ${}^L\De(u_1)$). Suppose that $c$ is not the distinguished right vertex of $\De(u_1)$. Let $\de$ be the (bi-rooted) graph obtained from $\De(u_1)$ by changing the right root to $c$. Then $\de=\De(v_1)$ for some $v_1$ for which $\mathbf{PS}\models u_1\Rc v_1$. Then $\mathbf{PS}\models u_1\wedge u_2\approx v_1\wedge u_2$ and we may continue with $v_1\wedge u_2$ instead of $u_1\wedge u_2$. Or, in other words, we may assume that $c$ is the distinguished right vertex of $\De(u_1)$. Since the label of $b$ as well as of $c$ is $y$ it follows that $\r(u_1)=\r(u_2)=y$. Using the notation introduced at the beginning of the present subsection we conclude
$$u_1=(t_kt_{k-1}\dots t_1y\wedge)\mbox{ and }u_2=(s_ls_{l-1}\dots s_1y\wedge)$$
for certain words $t_k,\dots,t_1,s_l,\dots,s_1\in F_2(X)$. According to Corollary \ref{idforfolding},
\begin{equation}\label{criticalequation}
\mathbf{PS}\models (t_k\dots t_1y\wedge)\wedge(s_l\dots s_1y\wedge)\approx(t_k\dots t_1s_l\dots s_1y\wedge) =:v.\end{equation}
Note that $v=(t_k\dots t_1y\wedge)(y\to (s_l\dots s_1y\wedge))$ (the $y$ to be substituted being the rightmost letter in $(t_k\dots t_1y\wedge)$);
by the description of $\De(w(s\to t))$ as it is given in section \ref{algebraB(X)} it follows that $\De(v)$ can be obtained by replacing in $\De(u_1)=\De((t_k\dots t_sy\wedge))$ the right distinguished vertex with the tree $\De(u_2)^R=\De((s_l\dots s_1y\wedge))^R$. However, the same graph is obtained by the edge-folding (\ref{folding}), that is, $\De(v)$ is obtained from $\De(u)$ by an edge-folding of the form (\ref{folding}). Finally, if $u'$ is any word such that $\De(u')$ is obtained from $\De(u)$ by the folding \ref{folding} then $\De(u')=\De(v)$ and so  $\mathbf{PS}\models u'\approx v$. The identity (\ref{criticalequation}) then implies $\mathbf{PS}\models u\approx u'$.
\end{proof}

Altogether we have proved the main result of the present section.
\begin{Thm} The binary algebra $(\A(X),\wedge)$ is a model of the free pseudosemilattice generated by $X$.
\end{Thm}

\subsection{The relations ${\le}_{\Rc}$, ${\le}_{\Lc}$, $\le$, $\Rc$, $\Lc$, and ${\Rc}\vee{\Lc}$ on  $\A(X)$}\label{greensrelations} It is a well known fact in semigroup theory that in any pseudosemilattice $E$, the subsets $(e]_{\Rc}$, $(e]_{\Lc}$ and $(e]_{\le}$ constitute respectively right normal bands (idempotent semigroups satisfying the identity $xyz\approx yxz$), left normal bands (idempotent semigroups satisfying the identity $xyz\approx xzy$) and semilattices. We leave these conclusions registered in the following lemma for future reference, but we would like to point out that they can be easily obtained also directly from the axioms for \pseudo s.

\begin{Lemma}\label{nband}
For each $e$ in a \pseudo\ $E$, the subsets $(e]_{\Rc}$, $(e]_{\Lc}$ and $(e]_{\le}$ are respectively a right normal band, a left normal band and a semilattice.
\end{Lemma}

We shall describe next the maximal subsemilattices and the maximal right/left normal subbands of $\A(X)$. Let $x\in X$ and $\ga\in\A(X)$. Then $\Te(x)\wedge\ga=\ga$ if $\l(\ga)=x$ (one can perform an edge-folding on $\Te(x)\sqcap\ga$ merging $\lv_{\Te(x)}$ with $\lv_\ga$, and then a deletion of the thorn $\{(\lv_{\Te(x)},\rv_{\Te(x)}),\rv_{\Te(x)}\}$; we obtain $\ga$ with this procedure). But on the other hand, if $\Te(x)\wedge\ga=\ga$ then $\l(\ga)=\l(\Te(x)\wedge\ga)=x$. Hence $$(\Te(x)]_{\Rc}=\{\al\in\A(X)\mid \l(\al)=x\}=: \R_x(X)$$ and by symmetry $$(\Te(x)]_{\Lc}=\{\al\in\A(X)\mid \r(\al)=x\}=: \L_x(X).$$ One can get now that
\[(\Te(x\wedge y)]_\le =\R_x(X)\cap\L_y(X)= \{\al\in\A(x)\mid (\l(\al),\r(\al))=(x,y)\}=:\S_{x,y}(X)\]
for $x,y\in X$.
\begin{Prop}
The sets $\R_x(X)$ and $\L_x(X)$ for $x\in X$ are respectively
the maximal right and left normal subbands of $\A(X)$, while the sets $\S_{x,y}(X)$ for $x,y\in X$ are the maximal subsemilattices of $\A(X)$.
\end{Prop}

\begin{proof}
It follows from $\l(\al)=\l((\al\wedge \be)\wedge\al)$ and $\l(\be)= \l((\be\wedge\al) \wedge\al)$ that any right normal subband of $\A(X)$ is contained in some $\R_x(X)$. But by Lemma \ref{nband} the sets $\R_x(X)$, $x\in X$, are right normal subbands of $\A(X)$ and hence are the maximal right normal subbands. The `left' case follows by symmetry. Since $(\l(\al\wedge\be),\r(\al\wedge\be))=(\l(\al),\r(\be))$ it is evident that any subsemilattice of $\A(X)$ is contained in some $\S_{x,y}(X)$. Once again by Lemma \ref{nband} we conclude that the sets $\S_{x,y}(X)$ are the maximal subsemilattices of $\A(X)$.
\end{proof}

The following corollary is now obvious.
\begin{Cor}\label{comA}
Two elements $\al,\be\in \A(X)$ commute if and only if $$(\l(\al),\r(\al))=(\l(\be),\r(\be)).$$
\end{Cor}

Deletions of thorns and edge-foldings may be applied in any order to a graph $\ga\in \B(X)$ in order to obtain $\ol\ga$. However, the application of an edge-folding may produce a new possibility to delete a thorn while the deletion of a thorn will never give rise to a new possibility for an edge-folding. It follows that, in order to obtain the reduced form $\ol\ga$ we may first perform all possible edge-foldings and delete only afterwards all non-essential thorns. Let $\wt\ga$ be the result obtained by applying to $\ga$ all possible edge-foldings. Then $\ol\ga$ is obtained from $\wt\ga$ by the deletion of (all non-essential) thorns.

Each edge-folding is a graph-homomorphism in the usual sense (a mapping sending vertices to vertices, edges to edges and preserving the adjacency relation). In particular, there is a graph homomorphism $\pi_\ga:\ga\twoheadrightarrow \wt\ga$. It follows immediately from the definition of edge-folding that for any subtree $\be$ of $\ga$ we have $\pi_\be=\pi_\ga\vert_\be$, and, in particular, $\wt\be$ is always a subtree of $\wt\ga$. (Note that $\ol\be$ is not necessarily a subtree of $\ol\ga$ since there may be a non-essential thorn in $\wt\ga$ which is essential in $\be$, hence may not be deleted in the reduction $\wt\be\to \ol\be$, yet it may be deleted in the reduction $\wt\ga\to\ol\ga$.)

In the following we shall consider four types of bipartite trees: bi-rooted, left-rooted, right-rooted and non-rooted ones. In this context, the terms ``subgraph'' and ``subtree'' have to be always understood within the appropriate category. For example, a left-rooted subtree of a left-rooted tree $\ga$ is a subtree $\be$ of $\ga$ with a distinguished left vertex which coincides with the left root of $\ga$. Take $\al,\be\in \A(X)$ and let us look carefully at the mapping $\al\sqcap\be\twoheadrightarrow \wt{\al\sqcap\be}$  since this case is crucial for the following results. The first observation is that $\pi_\al$ embeds $\al$ into $\ga=\wt{\al\sqcap\be}$ (since $\al=\wt\al$) with $\lv_\al\pi_\al =\lv_\ga$. So we can see ${}^L\al$ as a left-rooted subtree of ${}^L\ga$. By symmetry we can see also $\be^R$ as a right-rooted subtree of $\ga^R$. Furthermore, each non-essential thorn $\{e,a\}$ of $\ga$ is contained in $\al$ or $\be$. If
$\{e,a\}\subseteq\al$, then $\{e,a\}$ is also a thorn of $\al$ (the degree of the vertex $a$ in $\al$ is not greater than its degree in $\ga$), whence it is the only possible (essential) thorn of $\al$ since $\al\in\A(X)$. Thus $a=\rv_\al$ (since $\lv_\al=\lv_\ga$ and $\{e,a\}$ is a non-essential thorn of $\ga$) and $\r(\al)=\l(\al) \neq \r(\be)$ (otherwise we could merge $\rv_\al$ with $\rv_\be=\rv_\ga$ in $\ga$ by an edge-folding). In a dual way we can see that if $\{e,a\}\subseteq\be$, then $\{e,a\}$ is the only possible  (essential) thorn of $\be$ with $a=\lv_\be$ and $\l(\be)=\r(\be)\neq \l(\al)$. Summing it up, there are only two candidates for non-essential thorns in $\ga=\wt{\al\sqcap\be}$, namely $\{(\lv_\ga, \rv_\al),\rv_\al\}$ and $\{(\lv_\be, \rv_\ga),\lv_\be\}$ (the former if $\{(\lv_\al, \rv_\al),\rv_\al\}$ is an essential thorn of $\al$ and $\r(\al)\neq \r(\be)$, and the latter if $\{(\lv_\be, \rv_\be),\lv_\be\}$ is an essential thorn of $\be$ and $\l(\be)\neq \l(\al)$).

For $\al\in \A(X)$ let
\[{}^l\al=\begin{cases} {}^L\al\setminus\{(\lv_\al,\rv_\al),\rv_\al\} &\text{ if }\{(\lv_\al,\rv_\al),\rv_\al\}\text{ is a thorn }\\ {}^L\al &\text{ otherwise}\end{cases}
\]
and define $\al^r$ dually. If $\al=$\tikz[baseline=2.25pt,scale=0.3]{\coordinate [label=below:$x$] (4) at (0,0.5); \coordinate [label=below:$x$] (5) at (2.5,0.5);\draw[double] (4) --(5);\draw (4) node{$\bullet$};\draw (5) node{$\bullet$};} for some $x\in X$ then ${}^l\al$ is the singleton graph $\underset{x}{\bullet}$  considered as a bipartite graph with one (distinguished) left vertex and no right vertex; the dual is assumed for $\al^r$.
In any case, ${}^l\al$ and $\al^r$ are
left-rooted respectively right-rooted bipartite
trees without non-essential thorns.
We further note that unless $\al=$\tikz[baseline=2.25pt,scale=0.3]{\coordinate [label=below:$x$] (4) at (0,0.5); \coordinate [label=below:$x$] (5) at (2.5,0.5);\draw[double] (4) --(5);\draw (4) node{$\bullet$};\draw (5) node{$\bullet$};} for some $x\in X$ we always have that at least one of the two equalities ${}^l\al={}^L\al$ and $\al^r=\al^R$ holds. An immediate consequence is:
\begin{Lemma}\label{LRbelowimpliesbelow} Let $\al,\be\in \A(X)$; if ${}^l\al$ is a left-rooted subgraph of ${}^l\be$ and $\al^r$ is a right-rooted subgraph of $\be^r$ then $\al$ is a bi-rooted subgraph of $\be$.
\end{Lemma}

Moreover, for all $\al,\be\in\A(X)$ the above arguments imply that
\begin{equation}\label{productinFPSviafoldings}
\al\wedge \be={({}^l\al\cup\{(\lv_\al,\rv_\be)\}\cup \be^r)}{}\wt\ .
\end{equation}
 It follows that ${}^l\al$ is a left-rooted subtree of  ${}^l(\al\wedge\be)$ and $\be^r$ is a right-rooted subtree of $(\al\wedge\be)^r$.
\begin{Thm}\label{leonA} Let $\al,\be\in\A(X)$; then
\begin{enumerate}
\item $\be\wr\al$ if and only if ${}^l\al$ is a left-rooted subtree of ${}^l\be$;
\item $\be\wl\al$ if and only if $\al^r$ is a right-rooted subtree of $\be^r$;
\item $\be\le\al$ if and only if $\al$ is a bi-rooted subtree of $\be$.
\end{enumerate}
\end{Thm}
\begin{proof} (1) We have seen for arbitrary $\al,\be\in \A(X)$ that ${}^l\al$ is a left-rooted subgraph of ${}^l(\al\wedge\be)$. Hence, if $\al\wedge\be=\be$ then ${}^l\al$ is a left-rooted subgraph of ${}^l\be$. Suppose conversely that ${}^l\al$ is a left-rooted subgraph of ${}^l\be$; then a proof similar to that of Lemma \ref{idempotency} applied to formula (\ref{productinFPSviafoldings}) shows that $\al\wedge\be=\be$. The proof of item (2) is completely analogous. (3) Suppose first that $\al$ is a bi-rooted subgraph of $\be$;
then ${}^l\al$ is a left-rooted subgraph of ${}^l\be$ and $\al^r$ is a right-rooted subgraph of $\be^r$. It follows that $\be\wr\al$ as well as $\be\wl\al$ hence also $\be\le\al$. Suppose conversely that $\be\le\al$. Then $\be\wr\al$ and $\be\wl\al$ and hence ${}^l\al$ is a left-rooted subgraph of ${}^l\be$ and $\al^r$ is a right-rooted subgraph of $\be^r$. Lemma \ref{LRbelowimpliesbelow} now implies that $\al$ is a bi-rooted subgraph of $\be$
\end{proof}
An immediate consequence is:
\begin{Cor}\label{RonA}
For $\al,\be\in \A(X)$ the following hold.
\begin{enumerate}
\item $\al\Rc\be$ if and only if ${}^l\al={}^l\be$,
\item $\al\Lc\be$ if and only if $\al^r=\be^r$.
\end{enumerate}
\end{Cor}
Call a vertex in a vertex-labeled bipartite graph \emph{paired} if it is adjacent to another vertex having the same label. We can use the preceding results to determine the cardinalities of the $\Rc$- and $\Lc$-classes of $\A(X)$.
\begin{Cor}\label{LRsizes} For $\al\in \A(X)$ the size of the $\Rc$-class of $\al$ is the degree of $\lv_\al$ if $\lv_\al$ is paired, otherwise it is the degree of $\lv_\al$ plus $1$. Dually, the size of the $\Lc$-class of $\al$ is the degree of $\rv_\al$ if $\rv_\al$ is paired, otherwise it is the degree of $\rv_\al$ plus $1$.
\end{Cor}
\begin{proof} Each edge of $\al$ containing $\lv_\al$ determines a right vertex in ${}^l\al$ that can be chosen as a distinguished right vertex in order to get an $\Rc$-related element of $\A(X)$. If $\lv_\al$ is paired these are all possible choices for right roots of an $\Rc$-related element. If $\lv_\al$ is not paired then a further $\Rc$-related element can be obtained by adding a thorn to $\lv_\al$ and declaring the added vertex to be
the distinguished right vertex. The `$\Lc$-case' is completely analogous.
\end{proof}
A particular instance of this corollary is when $\al\in \A(X)$ has a thorn: if  $\lv_\al$/$\rv_\al$ is the degree $1$ vertex of a thorn of $\al$ then the $\Rc$-class/$\Lc$-class of $\al$ is a singleton.

Let us turn now to the relation $\Rc\vee\Lc$ on $\A(X)$. We shall say that two elements $\al$ and $\be$ of $\A(X)$ are \emph{connected} if $\al \mathrel{(\Rc\vee\Lc)} \be$, and we shall call the $\Rc\vee\Lc$-classes the \emph{connected components} of $\A(X)$. Thus $\al$ and $\be$ are connected if (and only if) there exists a finite sequence $\al=\ga_0,\ga_1, \cdots,\ga_n=\be$ of alternately $\Rc$- or $\Lc$-equivalent elements of $\A(X)$ (such a kind of sequence is referred to as an `$E$-chain' in semigroup literature). It follows from Corollary \ref{RonA} that the singleton sets $\{\text{\tikz[baseline=2.25pt,scale=0.3]{\coordinate [label=below:$x$] (4) at (0,0.5); \coordinate [label=below:$x$] (5) at (2.5,0.5);\draw[double] (4) --(5);\draw (4) node{$\bullet$};\draw (5) node{$\bullet$};}}\}$ are examples of connected components of $\A(X)$ and that they are the only connected components of $\A(X)$ with a single element. Note that the connectedness relation is an analogue for pseudosemilattices of Green's $\Dc$-relation for semigroups. Indeed, two members of a pseudosemilattice $E$ are connected if and only if they are $\Dc$-related in each idempotent generated locally inverse semigroup $S$ having $E$ as its pseudosemilattice of idempotents.

Let $\C'(X)$ denote the set of all finite non-trivial bipartite trees the vertices of which are labeled by letters of $X$ and which are thorn-free and reduced for edge foldings,
and let $\C(X)=\C'(X)\cup X$ where the letters $x\in X$ are represented as the singleton graphs $\underset{x}{\bullet}$. We need the following definition to characterize the connected components of $\A(X)$. For $\al=$ \tikz[baseline=2.25pt,scale=0.3]{\coordinate [label=below:$x$] (4) at (0,0.5); \coordinate [label=below:$x$] (5) at (2.5,0.5);\draw[double] (4) --(5);\draw (4) node{$\bullet$};\draw (5) node{$\bullet$};} set $\wh\al=\underset{x}{\bullet}$, otherwise, for each bi-rooted, left-rooted or right-rooted tree $\al$ let $\wh\al$ be the (non-rooted) bipartite tree obtained from $\al$ by un-marking the distinguished root(s) and removing the existing thorns. The mapping $\al\mapsto \wh\al$ is a surjection  $\A(X)\twoheadrightarrow\C(X)$. By construction
\begin{equation}\label{hatal}
\wh{{\,}^l\al}=\wh\al=\wh{\al^r}.
\end{equation}
The next result describes the connected components of $\A(X)$ as the equivalence classes induced by the mapping $\al\mapsto\wh\al$.

\begin{Prop}
For $\al,\be\in\A(X)$, $\al \mathrel{(\Rc\vee\Lc)}\be$ if and only if $\wh\al=\wh\be$.
\end{Prop}

\begin{proof}
The `only if' part follows from Corollary \ref{RonA} and (\ref{hatal}) since $\wh\al=\wh\be$ if $\al\Rc\be$ or $\al\Lc\be$. For the `if' part we may assume that $\wh\al=\wh\be$ has more than one vertex (if $\wh\al=\underset{x}{\bullet}=\wh\be$ then $\al=$\tikz[baseline=2.25pt,scale=0.3]{\coordinate [label=below:$x$] (4) at (0,0.5); \coordinate [label=below:$x$] (5) at (2.5,0.5);\draw[double] (4) --(5);\draw (4) node{$\bullet$};\draw (5) node{$\bullet$};}$=\be$). Suppose that $\al$ has a thorn $\{e,a\}$ with $b$  the other endpoint of $e$. Since $\wh\al\neq\underset{x}{\bullet}$ there exists a vertex $c\in\al$ distinct from $a$ and connected to $b$ by an edge. If we consider now $\al_1=\al\setminus\{e,a\}$ with distinguished vertices $b$ and $c$, then $\wh\al=\wh{\al_1}$ and $\wh{\al_1}$ is just the graph $\al_1$ with its distinguished vertices un-marked. By Corollary \ref{RonA} $\al\Rc\al_1$ or $\al\Lc\al_1$. In other words, we can assume that $\al$ has no thorn. In the same way we can assume also that $\be$ has no thorn. Thus we may assume that $\lv_\al,\rv_\al,\lv_\be,\rv_\be$ are all contained in $\wh\al$. Let $a_0,a_1,\cdots,a_n,a_{n+1}$ be the geodesic path in $\wh\al$ with $a_0$ and $a_1$ the distinguished vertices of $\al$ and $a_n$ and $a_{n+1}$ the distinguished vertices of $\be$. For $i=0,1,\cdots,n$ let $\ga_i$ be the graph $\wh\al$ but with $a_i$ and $a_{i+1}$ as distinguished vertices. Thus $\al=\ga_0$, $\ga_n=\be$, and $\ga_i\Rc\ga_{i+1}$ or $\ga_i\Lc\ga_{i+1}$ by Proposition \ref{RonA}. Hence $\al \mathrel{(\Rc\vee\Lc)}\be$.
\end{proof}

\begin{Cor}
For $\al\in\A(X)$, let $p$ and $q$ be respectively the number of non-paired vertices and the number of edges of $\al$. Then the connected component of $\A(X)$ containing $\al$ has $p+q$ elements.
\end{Cor}

\begin{proof} Let $\al\in\A(X)$; there are $q$ possibilities to choose a pair of distinguished vertices and $p$ possibilities to add a thorn and then choose another pair of distinguished vertices. All these choices lead to elements connected with $\al$ and all elements connected with $\al$ are thereby obtained. Hence the number of elements in the connected component containing $\al$ is $p+q$.
\end{proof}

For each $\ga\in\C(X)$ let $\A_\ga(X)$ be the connected component of $\A(X)$ represented by $\ga$, that is,
$$\A_\ga(X)=\{\al\in\A(X)\mid \wh\al=\ga\}.$$
The set $\C(X)$ is  partially ordered by reverse inclusion; for $\ga,\de\in\C(X)\setminus X$ define  $\de\unlhd\ga$ if (and only if) $\ga$ is a bipartite subgraph of $\de$ (meaning that the right/left vertices of $\ga$ are right/left vertices of $\de$); let further $\de\unlhd\underset{x}{\bullet}$ if and only if $\de$ contains a singleton subgraph $\underset{x}{\bullet}$ (independently of whether it is a left or right vertex).

\begin{Prop}
Let $\al,\be\in\A(X)$ and $\ga,\de\in\C(X)$.
\begin{enumerate}
\item If $\al\le\be$ then $\wh\al\unlhd\wh\be$.
\item If $\de\unlhd\ga$ then for each $\be\in\A_\ga(X)$ there exists  $\al\in\A_\de(X)$ such that $\al\le\be$.
\end{enumerate}
\end{Prop}

\begin{proof}
(1) follows from Theorem \ref{leonA}. For (2)
we may assume that $\de$ has more than one vertex (if $\de$ has only one vertex then $\de=\underset{x}{\bullet}=\ga$ and $\al=$\tikz[baseline=2.25pt,scale=0.3]{\coordinate [label=below:$x$] (4) at (0,0.5); \coordinate [label=below:$x$] (5) at (2.5,0.5);\draw[double] (4) --(5);\draw (4) node{$\bullet$};\draw (5) node{$\bullet$};}$=\be$). Hence let $\de\unlhd\ga$ and $\be\in\A_\ga(X)$ where $\de$ is a graph with more than one vertex. Let $\ga'$ be a subgraph of $\de$ isomorphic with $\ga$. If $\be$ has no thorn then let $(a,b)$ be the edge of $\ga'$ corresponding to the edge $(\lv_\be,\rv_\be)$ of $\be$, and set $\al\in\A_\de(X)$ to be the graph $\de$ with distinguished vertices $a$ and $b$. The graph $\be$ is then a bi-rooted subgraph of $\al$, and $\al\le\be$ by Proposition \ref{leonA}. If $\be$ has a thorn $\{e,a\}$ then let $b$ be the other endpoint of $e$ and let $b'$ be the vertex of $\ga'$ corresponding to $b$. Define (i) $\al$ to be $\de$ with distinguished vertices $a'$ and $b'$ if there exists a vertex $a'\in\de$, adjacent to $b'$ and having the same label as $b'$, or (ii) attach in $\de$ a thorn $\{e',a'\}$ to $b'$ otherwise and let  $\al$ be $\de\cup\{e',a'\}$ with distinguished vertices $a'$ and $b'$. Then by construction $\al\in\A_\de(X)$ and $\be$ is a bi-rooted subtree of $\al$, whence $\al\le\be$ and we have shown (2).
\end{proof}
The arguments in this proof also show that the number of elements $\al$ that are below a given $\be\in \A_\ga(X)$ essentially is the number of distinct realizations of $\ga$ as a subgraph of $\de$. In case $\ga=\underset{x}{\bullet}$ the latter has to be modified slightly: instead of taking the number of all vertices in $\de$ having label $x$ one has to take the number of all pairs of paired vertices having label $x$ and plus the number of all  non-paired vertices with label $x$ (that is, paired vertices only count half).

\section{A basis for the variety $\mathbf{SPS}$}\label{basisSPS}   A transparent combinatorial description for the identities $u\approx v$ which hold in each strict pseudosemilattice has been obtained (via semigroup theoretic methods) by the first author \cite{wpstrict}. In order to formulate it we need yet another combinatorial invariant of words $u\in F_2(X)$: the \emph{$2$-content} $\co_2(u)$ is defined inductively by setting $\co_2(x):=\{(x,x)\}$ for each letter $x\in X$ and letting
$$\co_2(u\wedge v):=\co_2(u)\cup\{(\l(u),\r(v))\}\cup\co_2(v)$$  for arbitrary $u,v\in F_2(X)$ (and we also extend this definition to members of $\B(X)$ in the obvious way). Then, for any $u,v\in F_2(X)$,
\begin{equation}\label{wpsps}\mathbf{SPS}\models u\approx v\mbox{ if and only if }\l(u)=\l(v),\ \co_2(u)=\co_2(v), \r(u)=\r(v).
\end{equation}

Let $\rho_{\mathbf{SPS}}$ be the fully invariant congruence on $\A(X)$ corresponding to $\mathbf{SPS}$ (that is, $\A(X)/\rho_{\mathbf{SPS}}$ is the free strict pseudosemilattice on $X$). Then (\ref{wpsps}) implies that
$$\rho_{\mathbf{SPS}}=\{(\al,\be)\in \A(X)\times \A(X)\mid (\l(\al),\co_2(\al),\r(\al))=(\l(\be),\co_2(\be),\r(\be))\}.$$
In this section we intend to obtain an identity basis for the variety $\mathbf{SPS}$; this is equivalent to find a (nice) subset of $\A(X)\times \A(X)$ ($X$ a countably infinite set) which generates $\rho_{\mathbf{SPS}}$ as a fully invariant congruence.

Let $\rho,\si\subseteq \A(X)\times \A(X)$; we say that $\rho$ \emph{is a consequence of} $\si$ if $\rho$ is contained in the fully invariant congruence generated by $\si$. Two relations $\rho$ and $\si$ are \emph{equivalent} if they generate the same fully invariant congruence. The following observation allows for technical simplifications.
\begin{Lemma} Let $u=u(x_1,\dots,x_n),v=v(x_1,\dots,x_n)\in F_2(X)$ (the notation indicating that only the variables $x_1,\dots,x_n$ occur in $u$ and $v$). Let $x_1',\dots,x_n'$ be new variables distinct from all $x_1,\dots,x_n$. Then, for the class of all idempotent binary algebras,  the two identities
\begin{enumerate}
\item $u(x_1,\dots,x_n)\approx v(x_1,\dots,x_n)$,
\item $u(x_1\wedge x_1',\dots,x_n\wedge x_n')\approx v(x_1\wedge x_1',\dots,x_n\wedge x_n')$
\end{enumerate}
are equivalent.
\end{Lemma}
\begin{proof} We may substitute $x_i\wedge x_i'$ for $x_i$ to obtain the second from the first identity. For the converse, substitute  $x_i$ for  $x_i'$ and  apply the idempotent law $x\wedge x\approx x$.
\end{proof}

The words of the form $w(x_1\wedge x_1',\dots,x_n\wedge x_n')$ have disjoint left and right contents. In particular, each pair $(\al,\be)\in \A(X)\times \A(X)$ is equivalent to a pair $(\al',\be')$ for which $\co_l(\al')\cap\co_r(\al')=\emptyset=\co_l(\be')\cap \co_r(\be')$. By Corollary \ref{comA}, for each pair $(\al,\be)\in \rho_{\mathbf{SPS}}$ the equality $\al\wedge\be=\be\wedge \al$ holds and hence the relations $\al\ge\al\wedge\be\le\be$ are satisfied. Since all involved graphs are finite, for any $\al,\be$ with $\al\wedge\be=\be\wedge \al$ the intervals $[\al\wedge\be,\al]$ and $[\al\wedge\be,\be]$ are finite as well. Hence there exist maximal chains
$$\al=\al_0\succ\al_1\succ\cdots\succ\al_{m-1}\succ\al_m=\al\wedge\be$$ and $$\al\wedge\be=\be_n\prec\be_{n-1}\prec\cdots\be_1\prec\be_0=\be$$
where $\al_i\succ \al_{i+1}$ means that $\al_i$ covers $\al_{i+1}$ while $\be_j\prec \be_{j-1}$ means that $\be_j$ is covered by $\be_{j-1}$. In this situation we have for any congruence $\rho$ on $\A(X)$ that $\al\mathrel{\rho}\be$ if and only if $\al\mathrel{\rho}\al\wedge\be\mathrel{\rho}\be$ and the latter is equivalent to $\al_i\mathrel{\rho}\al_{i+1}$ and $\be_j\mathrel{\rho}\be_{j-1}$ for all $i$ and $j$. Altogether we may state that each pair $(\al,\be)$ (with $\al\wedge\be=\be\wedge \al$) is equivalent to a finite set of pairs $(\al_i,\al_{i+1})$ with $\al_i\succ\al_{i+1}$. Note that if $\al\succ\be$ and $\be$ has disjoint left and right contents then there exists exactly one  vertex in $\be$ that is not in $\al$ (this vertex has degree $1$). We are going to consider a special sort of such covering pairs. A pair $(\al,\be)$ of commuting elements of $\A(X)$ is called \emph{elementary} if
\begin{enumerate}
\item $\co_2(\al)=\co_2(\be)$ and the left and right contents of $\be$ (and therefore of $\al$) are disjoint,
\item $\al$ covers $\be$,
\item the unique degree $1$ vertex in $\be\setminus\al$ is adjacent to a distinguished vertex of $\be$.
\end{enumerate}
Note that if $\al\in\A(X)$ has disjoint left and right contents then the replacement of the pair of distinguished  vertices by any other pair of adjacent vertices leads to another member of $\A(X)$ (since no such change will lead to a non-essential thorn). The next lemma in combination with all preceding remarks  essentially shows that elementary pairs will be sufficient to generate the fully invariant congruence $\rho_{\mathbf{SPS}}$.
\begin{Lemma}\label{changingroots} Let $\be<\al$ and suppose that the right and the left contents of $\be$ are disjoint; let $\be'$ and $\al'$ be obtained from $\be$ and $\al$ by moving the distinguished vertices (again to the same vertices for $\be'$ and $\al'$). Then the pairs $(\al,\be)$ and $(\al',\be')$ are equivalent. It follows that each covering pair $(\al,\be)\in\rho_{\mathbf{SPS}}$ where $\be$  has disjoint left and right contents is equivalent to an elementary pair.
\end{Lemma}
\begin{proof} Let $(\lv,\rv)$ be the pair of distinguished vertices of $\al$ and $\be$ and suppose as case (i)  that the distinguished right vertex $\rv$ is changed to $\rv'$ to obtain $\al'$ and $\be'$ (the left root being unchanged). Suppose that for the labels of the involved vertices we have
$x=\cb_\lv$, $y=\cb_\rv$ and $z=\cb_{\rv'}$. Let $\ga=$\tikz[baseline=2.25pt,scale=0.3]{\coordinate [label=below:$x$] (4) at (0,0.5); \coordinate [label=below:$z$] (5) at (2.5,0.5);\draw[double] (4) --(5);\draw (4) node{$\bullet$};\draw (5) node{$\bullet$};} and $\de=$\tikz[baseline=2.25pt,scale=0.3]{\coordinate [label=below:$x$] (4) at (0,0.5); \coordinate [label=below:$y$] (5) at (2.5,0.5);\draw[double] (4) --(5);\draw (4) node{$\bullet$};\draw (5) node{$\bullet$};}; then
$$\al\wedge \ga=\al',\ \be\wedge\ga=\be'\mbox{ and } \al'\wedge\de=\al,\ \be'\wedge \de=\be$$
and hence $(\al,\be)$ and $(\al',\be')$ are equivalent.

If, in case (ii),  $\lv$ is changed to $\lv'$ and $\rv$ is kept unchanged then for $z=\cb_{\lv'}$ and $\ga=$\tikz[baseline=2.25pt,scale=0.3]{\coordinate [label=below:$z$] (4) at (0,0.5); \coordinate [label=below:$y$] (5) at (2.5,0.5);\draw[double] (4) --(5);\draw (4) node{$\bullet$};\draw (5) node{$\bullet$};} and keeping the rest of the notation the same we obtain
$$\ga\wedge \al=\al',\ \ga\wedge\be=\be'\mbox{ and } \de\wedge\al'=\al,\ \de\wedge\be'=\be$$
so that again $(\al,\be)$ and $(\al',\be')$ are equivalent.

Finally, if both vertices $(\lv,\rv)$ are changed to $(\lv',\rv')$ to get $\al'$ and $\be'$, respectively, then we choose the minimal subtree containing both edges. This subtree gives rise to a sequence of changes, alternately  of types (i) and (ii), which eventually transform $\al$ to $\al'$ and $\be$ to $\be'$. In each of these transformations one obtains a pair that is equivalent to the preceding one. Altogether, the pairs $(\al,\be)$ and $(\al',\be')$ are equivalent.

In order to prove the last statement of the lemma we just have to choose the distinguished vertices of $\al'$ and $\be'$ appropriately: the unique vertex in $\be\setminus\al$ must be connected in $\be$ by an edge to one of the two chosen distinguished vertices of $\be'$.
\end{proof}
Summing up the results of this section so far we get the first main result.
\begin{Thm}\label{reductiontoelementary} \begin{enumerate}
\item Each fully invariant congruence $\rho$ on $\A(X)$  contained in $\rho_{\mathbf{SPS}}$ is generated by elementary pairs. \item The fully invariant congruence $\rho_{\mathbf{SPS}}$ is generated by all elementary pairs.
\end{enumerate}
\end{Thm}

Next we show that the set of all elementary pairs is a consequence
of a set of very special elementary pairs. Let $(\al,\be)$ be an elementary pair and let $a$ be the unique vertex in $\be\setminus \al$; suppose its label is $\cb_a=z$. Let $\lv,\rv$ denote the distinguished left and right vertices of $\al$ (and $\be$) and suppose their labels are $x$ and $y$. Note that $x,y,z$ are pairwise distinct. The subtree of $\be$ spanned by the vertices $a,\lv,\rv$ is of the form either (i) \tikz[baseline=0pt,scale=0.3]{\coordinate [label=left:$x$] (1) at (0,0.5);\coordinate [label=right:$z$] (2) at (2,1);
\coordinate[label=right:$y$] (3) at (2,0);  \draw (2,1) -- (0,0.5);\draw[double] (0,0.5) --(2,0); \draw (0,0.5) node{$\bullet$}; \draw (2,1) node{$\bullet$}; \draw (2,0) node{$\bullet$};} or (ii) \tikz[baseline=0pt,scale=0.3]{\coordinate [label=left:$x$] (1) at (0,0);\coordinate [label=left:$z$] (2) at (0,1);
\coordinate[label=right:$y$] (3) at (2,0.5);  \draw (0,1) -- (2,0.5);\draw[double] (0,0) --(2,0.5); \draw (2,0.5) node{$\bullet$}; \draw (0,1) node{$\bullet$}; \draw (0,0) node{$\bullet$};}. Let $\al'$ be a minimal subtree of $\al$ containing $\lv,\rv$ and a pair $(c,d)$ of adjacent vertices with pair of labels $(x,z)$ in case (i) and $(z,y)$ in case (ii) (and such that $\lv$ and $\rv$ are the distinguished vertices of $\al'$). Let $\be'=\al'\cup\{a,e\}$ where $e=(\lv,a)$ in case (i) and $e=(a,\rv)$ in case (ii). Then $(\al',\be')$ is also an elementary pair. A straightforward calculation shows the next result.
\begin{Lemma}\label{reducingelementarypairs} Let $\al,\be,\al',\be'$ be defined as before. Then $\al\wedge \al'=\al$ and $\al\wedge\be'=\be$. In particular, $(\al,\be)$ is a consequence of $(\al',\be')$.
\end{Lemma}

For the graph $\be'$ in the situation described above, there are four possible types, namely, for some
$n\geq 2$,

\centerline{
\tikz[baseline=15pt,scale=0.5]{
\coordinate[label=above:$x_{2n}$] (1) at (0,2);
\coordinate[label=below:$x_1$] (2) at (1,0);
\coordinate[label=above:$x_2$] (3) at (2,2);
\coordinate[label=below:$x_3$] (4) at (3,0);
\coordinate[label=above:$x_{2n-2}$] (5) at (5,2);
\coordinate[label=below:$x_{2n-1}$] (6) at (6,0);
\coordinate[label=above:$x_{2n}$] (7) at (7,2);
\coordinate[label=below:$x_{1}$] (8) at (8,0);
\draw (4,1)node{$\dots$};
\draw (1) node {$\bullet$};
\draw (2) node {$\bullet$};
\draw (3) node {$\bullet$};
\draw (4) node {$\bullet$};
\draw (5) node {$\bullet$};
\draw (6) node {$\bullet$};
\draw (7) node {$\bullet$};
\draw (8) node {$\bullet$};
\draw (4)--(3.2,0.7);
\draw (4.8,1.3)--(5);
\draw (1)--(2);
\draw[double] (2)--(3);
\draw (3)--(4);
\draw (5)--(6)--(7)--(8);}
or
\tikz[baseline=15pt,scale=0.5]{
\coordinate[label=above:$x_{2n}$] (1) at (0,2);
\coordinate[label=below:$x_1$] (2) at (1,0);
\coordinate[label=above:$x_2$] (3) at (2,2);
\coordinate[label=below:$x_3$] (4) at (3,0);
\coordinate[label=above:$x_{2n-2}$] (5) at (5,2);
\coordinate[label=below:$x_{1}$] (6) at (6,0);
\coordinate[label=above:$x_{2n}$] (7) at (7,2);
\draw (4,1)node{$\dots$};
\draw (1) node {$\bullet$};
\draw (2) node {$\bullet$};
\draw (3) node {$\bullet$};
\draw (4) node {$\bullet$};
\draw (5) node {$\bullet$};
\draw (6) node {$\bullet$};
\draw (7) node {$\bullet$};
\draw (4)--(3.2,0.7);
\draw (4.8,1.3)--(5);
\draw (1)--(2);
\draw[double] (2)--(3);
\draw (3)--(4);
\draw (5)--(6)--(7);}
}
{\parindent=0pt with $(x_1,x_{2n})=(x,z)$ in} case (i) and

\centerline{
\tikz[baseline=15pt,scale=0.5]{
\coordinate[label=below:$x_{2n}$] (1) at (0,0);
\coordinate[label=above:$x_1$] (2) at (1,2);
\coordinate[label=below:$x_2$] (3) at (2,0);
\coordinate[label=above:$x_3$] (4) at (3,2);
\coordinate[label=below:$x_{2n-2}$] (5) at (5,0);
\coordinate[label=above:$x_{2n-1}$] (6) at (6,2);
\coordinate[label=below:$x_{2n}$] (7) at (7,0);
\coordinate[label=above:$x_{1}$] (8) at (8,2);
\draw (4,1)node{$\dots$};
\draw (1) node {$\bullet$};
\draw (2) node {$\bullet$};
\draw (3) node {$\bullet$};
\draw (4) node {$\bullet$};
\draw (5) node {$\bullet$};
\draw (6) node {$\bullet$};
\draw (7) node {$\bullet$};
\draw (8) node {$\bullet$};
\draw (4)--(3.2,1.3);
\draw (4.8,0.7)--(5);
\draw (1)--(2);
\draw[double] (2)--(3);
\draw (3)--(4);
\draw (5)--(6)--(7)--(8);}
or
\tikz[baseline=15pt,scale=0.5]{
\coordinate[label=below:$x_{2n}$] (1) at (0,0);
\coordinate[label=above:$x_1$] (2) at (1,2);
\coordinate[label=below:$x_2$] (3) at (2,0);
\coordinate[label=above:$x_3$] (4) at (3,2);
\coordinate[label=below:$x_{2n-2}$] (5) at (5,0);
\coordinate[label=above:$x_{1}$] (6) at (6,2);
\coordinate[label=below:$x_{2n}$] (7) at (7,0);
\draw (4,1)node{$\dots$};
\draw (1) node {$\bullet$};
\draw (2) node {$\bullet$};
\draw (3) node {$\bullet$};
\draw (4) node {$\bullet$};
\draw (5) node {$\bullet$};
\draw (6) node {$\bullet$};
\draw (7) node {$\bullet$};
\draw (4)--(3.2,1.3);
\draw (4.8,0.7)--(5);
\draw (1)--(2);
\draw[double] (2)--(3);
\draw (3)--(4);
\draw (5)--(6)--(7);}
}
{\parindent=0pt  with $(x_{2n},x_1)=(z,y)$ in case (ii).} In addition, the labels satisfy $x_i\ne x_{i+2}$ for all $i$ and
$$\{x_1,x_3,\dots,x_{2n-1}\}\cap\{x_2,x_4,\dots,x_{2n}\}=\emptyset.$$

Suppose now that the variables $x_1,\dots,x_{2n}$ ($n\geq 2$) are pairwise distinct and denote the resulting graphs (in this order) by $\be_n,\ga_n,\mu_n,\nu_n$, respectively. Moreover, denote by $\al_n,\ga_n',\mu_n',\nu_n'$, respectively, the trees obtained by removing from the respective graphs the degree $1$ vertex adjacent to a distinguished vertex (and the corresponding edge). Then all the pairs
\begin{equation}\label{elementarypairs}
(\al_n,\be_n),(\ga_n',\ga_n),(\mu_n',\mu_n),(\nu_n',\nu_n)
\end{equation}
are elementary and Lemma \ref{reducingelementarypairs} confirms that each elementary pair is a consequence of one of these pairs. Moreover, an easy observation is the following.
\begin{Lemma}\label{nplus1impliesn} For each $n\ge 2$,
\begin{enumerate}
\item $(\al_n,\be_n)$ is a consequence of $(\al_{n+1},\be_{n+1})$,
\item $(\ga_n',\ga_n)$ is a consequence of $(\ga_{n+1}',\ga_{n+1})$,
\item $(\mu_n',\mu_n)$ is a consequence of $(\mu_{n+1}',\mu_{n+1})$,
\item $(\nu_n',\nu_n)$ is a consequence of $(\nu_{n+1},\nu_{n+1})$.
\end{enumerate}
\end{Lemma}
\begin{proof} In cases (1) and (3) consider the substitution $x_{2n+2}\mapsto x_{2n},\ x_{2n+1}\mapsto x_1$, in cases (2) and (4) the substitution $x_{2n+2}\mapsto x_{2n},\ x_{2n-1}\mapsto x_1$ (and, in all cases, $x\mapsto x$ for all other $x$). The endomorphism induced by this substitution maps the respective $(n+1)$st pair to the $n$th pair.
\end{proof}

\begin{Lemma}\label{trivialreduction} For each $n\geq 2$, $(\ga_n',\ga_n)$ is a consequence of $(\al_n,\be_n)$ and $(\nu_n',\nu_n)$ is a consequence of $(\mu_n',\mu_n)$.
\end{Lemma}
\begin{proof} Consider the endomorphism $\psi:\A(X)\to \A(X)$ induced by the substitution $x_{2n-1}\mapsto x_1$ and $x\mapsto x$ for all $x\ne x_{2n-1}$. Then $(\ga_n',\ga_n)=(\al_n\psi,\be_n\psi)$ and $(\nu_n',\nu_n)=(\mu_n'\psi,\mu_n\psi)$.
\end{proof}
\begin{Lemma}\label{reductiontoB} For each  $n\ge 2$,
$(\mu_n',\mu_n)$ is a consequence of $(\al_{n+1},\be_{n+1})$.
\end{Lemma}
\begin{proof} Let $\al$ and $\be$ be obtained from $\al_{n+1}$ and $\be_{n+1}$ by moving the left root to the vertex labeled $x_3$ (and leaving the right root unchanged). By Lemma \ref{changingroots} the pairs $(\al,\be)$ and $(\al_{n+1},\be_{n+1})$
are equivalent. Now consider the endomorphism $\psi:\A(X)\to \A(X)$ defined by the substitution
\[
x_1\mapsto x_{2n},\  x_i \mapsto x_{i-1}
\text{ for }2\le i \le 2n+1,\ x_{2n+2}\mapsto x_1.
\]
Then $(\mu_n',\mu_n)=(\al\psi,\be\psi)$, hence $(\mu_n',\mu_n)$ is a consequence of $(\al,\be)$ and therefore of $(\al_{n+1},\be_{n+1})$.
\end{proof}
The main result of the present section is now a consequence of Theorem \ref{reductiontoelementary} and Lemmas \ref{reducingelementarypairs}, \ref{trivialreduction}, \ref{reductiontoB} and can be formulated as follows.
\begin{Thm}\label{ThmbasisSPS} The relation
$\{(\al_n,\be_n)\mid n\ge 2\}$
generates $\rho_{\mathbf{SPS}}$ as a fully invariant congruence.
\end{Thm}
Let $u_n,v_n\in F_2(X)$ be words such that $\al_n=\De(u_n)=\Te(u_n)$ and $\be_n=\De(v_n)=\Te(v_n)$.  Then Theorem \ref{ThmbasisSPS} may be reformulated as follows.
\begin{Thm}\label{basisSPSintermsofwords} The set $\{u_n\approx v_n\mid n\ge 2\}$ defines the variety $\mathbf{SPS}$ within the variety of all pseudosemilattices.
\end{Thm}
 We note that the words $u_n$ and $v_n$ are uniquely determined. Indeed, if in a graph $\ga\in\B(X)$ all vertices have degree at most $2$ then $\ga=\De(u)$ for a \textbf{unique} $u\in F_2(X)$. This follows by induction from the fact that for $u=u_1\wedge u_2$, if each vertex in $\De(u)$ has degree at most $2$ then the distinguished vertex in ${}^L\De(u_1)$ has degree at most
$1$ hence $\De(u_1)$ can be uniquely reconstructed from ${}^L\De(u_1)$. Likewise, $\De(u_2)$ can be uniquely reconstructed from $\De(u_2)^R$.

\section{The variety $\mathbf{SPS}$ is inherently non-finitely based}\label{INFB}
A locally finite variety $\mathbf{V}$ is \emph{inherently non-finitely based} if $\mathbf{V}$ is not contained in any finitely based locally finite variety. We are going to show that $\mathbf{SPS}$ admits this remarkable property.

Recall from subsection \ref{greensrelations} the definition of the graph homomorphism $\pi_\ga:\ga\twoheadrightarrow \wt\ga$ for $\ga\in \B(X)$.
A \emph{path} $p$ of length $k$ in a tree $\ga$ is a sequence $a_0,a_1,\dots,a_k$ of  vertices such that any two consecutive elements $a_i$ and $a_{i+1}$ are adjacent (and thus, in particular, are alternately left and right vertices). The image $p\pi_\ga$ of a path $p$ is a path in $\wt\ga$ of the same length as that of $p$.
Let us call an identity $u\approx v$ \emph{non-trivial (for $\mathbf{PS}$)} if it does not hold in $\mathbf{PS}$. Let $u,v,w\in F_2(X)$; then $w$ is an \emph{isoterm for the identity $u\approx v$ relative to $\mathbf{PS}$} if no non-trivial identity $w\approx w'$ is a consequence of $u\approx v$.

Now fix $n\in \mathbb{N}$ and let

\centerline{
\tikz[scale=0.5]{\draw(-1,1) node{$\la=$};
\coordinate[label=below:$x_1$] (1) at (0,0);
\coordinate[label=above:$x_2$] (2) at (1,2);
\coordinate[label=below:$x_3$] (3) at (2,0);
\coordinate[label=above:$x_4$] (4) at (3,2);
\coordinate[label=below:$x_{2n-1}$] (5) at (5,0);
\coordinate[label=above:$x_{2n}$] (6) at (6,2);
\coordinate[label=below:$x_{2n+1}$] (7) at (7,0);
\coordinate[label=above:$x_{2n+2}$] (8) at (8,2);
\draw (4,1)node{$\dots$};
\draw (1) node {$\bullet$};
\draw (2) node {$\bullet$};
\draw (3) node {$\bullet$};
\draw (4) node {$\bullet$};
\draw (5) node {$\bullet$};
\draw (6) node {$\bullet$};
\draw (7) node {$\bullet$};
\draw (8) node {$\bullet$};
\draw (4)--(3.2,1.3);
\draw (4.8,0.7)--(5);
\draw (1)--(2);
\draw (2)--(3)--(4);
\draw (5)--(6)--(7)--(8);}}

{\parindent=0pt and, for each $k\in \mathbb{N}$ let}

\centerline{\tikz[scale=0.5]{
\draw (-1.5,1) node{$\la_k=$};
\coordinate[label=below:$\la$] (1) at (1,0);
\coordinate[label=below:$\la$] (2) at (5,0);
\coordinate[label=below:$\la$] (3) at (10,0);
\draw (0,0) node{$\bullet$};
\draw (0.5,2) node{$\bullet$};
\draw (2,0) node{$\bullet$};
\draw (2.5,2) node{$\bullet$};
\draw (4,0) node{$\bullet$};
\draw (4.5,2) node{$\bullet$};
\draw (6,0) node{$\bullet$};
\draw (6.5,2) node{$\bullet$};
\draw (9,0) node{$\bullet$};
\draw (9.5,2) node{$\bullet$};
\draw (11,0) node{$\bullet$};
\draw (11.5,2) node{$\bullet$};
\draw[double] (0,0)--(0.5,2);
\draw (2,0)--(2.5,2)--(4,0)--(4.5,2);
\draw (6,0)--(6.5,2)--(7,1);
\draw (7.85,1)node{$\dots$};
\draw (8.5,1)--(9,0)--(9.5,2);
\draw (11,0)--(11.5,2);
\draw[dotted] (0,0)--(2,0);
\draw[dotted] (0.5,2)--(2.5,2);
\draw[dotted] (4,0)--(6,0);
\draw[dotted] (4.5,2)--(6.5,2);
\draw[dotted] (9,0)--(11,0);
\draw[dotted] (9.5,2)--(11.5,2);}
}
{\parindent=0pt where the segment $\la$ occurs $k$ times}. Let $m_k\in F_2(X)$ be such that $\De(m_k)=\la_k$.

\begin{Lemma}\label{isoterm} Each word $m_k$ is an isoterm for the identity $u_n\approx v_n$ (relative to  $\mathbf{PS}$).
\end{Lemma}
\begin{proof} Let $\psi$ be an endomorphism of $F_2(X)$ and $m\in F_2(X)$ be such that $\mathbf{PS}\models m_k\approx m$ and such that either $u_n\psi$ or $v_n\psi$ is a subword of $m$. We prove that this implies $\mathbf{PS}\models u_n\psi\approx v_n\psi$. Once this claim is proved the statement of the lemma is an immediate consequence. Indeed, this claim implies that there is no deduction process, using an identity of the form $u_n\psi\approx v_n\psi$,  which transforms $m_k$ to a word $u$ for which the identity $m_k\approx u$ is non-trivial.

Suppose that $u_n\psi$ is a subword of $m$ (the case $v_n\psi$ a subword of $m$ is similar and, in fact, simpler). First of all, $\De(u_n\psi)$ is a subtree of $\De(m)$. Let us consider the skeleton $\sk(u_n,\psi)$ which is a subtree of $\De(u_n\psi)$ and hence a subtree of $\De(m)$. For $i=1,\dots,2n$ let
\[y_i=\begin{cases}
\l(x_i\psi) &\text{ if }i\text{ is odd}\\
\r(x_i\psi) &\text{ if }i\text{ is even;}\end{cases}
\]
{\parindent=0pt then} $\sk(u_n,\psi)$ is the following graph

\centerline{\tikz[scale=0.5]{
\coordinate[label=below:$y_1$] (1) at (0,0);
\coordinate[label=above:$y_2$] (2) at (1,2);
\coordinate[label=below:$y_3$] (3) at (2,0);
\coordinate[label=above:$y_4$] (4) at (3,2);
\coordinate[label=above:$y_{2n-2}$] (5) at (5,2);
\coordinate[label=below:$y_{2n-1}$] (6) at (6,0);
\coordinate[label=above:$y_{2n}$] (7) at (7,2);
\coordinate[label=below:$y_{1}$] (8) at (8,0);
\draw (4,1)node{$\dots$};
\draw (1) node {$\bullet$};
\draw (2) node {$\bullet$};
\draw (3) node {$\bullet$};
\draw (4) node {$\bullet$};
\draw (5) node {$\bullet$};
\draw (6) node {$\bullet$};
\draw (7) node {$\bullet$};
\draw (8) node {$\bullet$};
\draw[double] (1)--(2);
\draw (2)--(3)--(4);
\draw (5)--(6)--(7)--(8);
\filldraw (8.4,1)circle (0.5pt);}}
{\parindent=0pt Let}  the vertices in this graph be denoted by $a_1,a_2,\dots,a_{2n},a_{2n+1}$
(in the obvious way). Note that $p=a_1,\dots,a_{2n+1}$ is a  path of length $2n$. Let us consider the respective graphs reduced for edge-foldings: $\wt{\sk(u_n,\psi)}$ is a subtree of $\wt{\De(u_n\psi)}$ which in turn is a subtree of $\wt{\De(m)}$. Let $\pi:\sk(u_n,\psi)\to \wt{\sk(u_n,\psi)}$ be the canonical graph homomorphism. The image path $ p\pi$ of $p$ has length $2n$ and is a  path in $\wt{\sk(u_n,\psi)}\subseteq \wt{\De(m)}$ starting at $a_1\pi$ and ending at $a_{2n+1}\pi$ with both vertices having the same label. Consider the graph $\vartheta\la_k$ obtained from $\la_k$ by attaching a thorn to each vertex. Then $\wt{\De(m)}$ is a subgraph of $\vartheta\la_k$ and $ p\pi$ can be considered as a path in $\vartheta\la_k$. However, each  path of even length in $\vartheta\la_k$ whose endpoints are distinct but have the same label must have length at least $2n+2$. It follows that the endpoints of $ p\pi$ must coincide, in particular $a_1\pi=a_{2n+1}\pi$. Let $\pi':\sk(v_n,\psi)\to \wt{\sk(v_n,\psi)}$ be the canonical homomorphism. As mentioned in subsection 3.3, $\pi'\vert_{\sk(u_n,\psi)}=\pi$. It follows that $a_1\pi'=a_{2n+1}\pi'$. An inspection of the graph $\De(v_n)$ then implies $a_0\pi'=a_{2n}\pi'$ where $a_0$ is the unique vertex in $\sk(v_n,\psi)\setminus \sk(u_n,\psi)$. But this means that
$\wt{\sk(u_n,\psi)}=\wt{\sk(v_n,\psi)}.$
From this we conclude that $\wt{\De(v_n\psi)}=\wt{\De(u_n\psi)}$. Indeed, take $\De(v_n\psi)$ and start to apply the edge-foldings that map the subtree $\sk(v_n,\psi)$ to $\wt{\sk(v_n,\psi)}$. In this process, $a_1$ is identified with $a_{2n+1}$ and $a_0$ is identified with $a_{2n}$ and so $\wt{\sk(v_n,\psi)}$ coincides with $\wt{\sk(u_n,\psi)}$. Next observe that in the resulting graph, on the vertex $a_0\pi'=a_{2n}\pi'$ two copies of $\De(x_{2n}\psi)^R$ are attached (meaning that the distinguished right vertices of these two copies are identified with the vertex $a_0\pi'=a_{2n}\pi'$). A sequence of edge-foldings then reduces the two copies eventually to one copy. But now we have arrived at a graph that can be obtained from $\De(u_n\psi)$ by edge-foldings. Altogether $\wt{\De(v_n\psi)}=\wt{\De(u_n\psi)}$ and thus also $\ol{\De(v_n\psi)}=\ol{\De(u_n\psi)}$ so that $\mathbf{PS}\models u_n\psi\approx v_n\psi$.
\end{proof}

Denote, for any pseudosemilattice $E$ and any $n\ge 2$, by $\rho_n$ the smallest congruence $\rho$ such that $E/\rho\models u_n\approx v_n$. The next statement is an immediate consequence of Lemma \ref{isoterm}.
\begin{Cor}\label{infb} The relatively free pseudosemilattice $\A(x_1,\dots,x_{2n+2})/\rho_n$ is infinite for each $n\ge 2$.\end{Cor}
We arrive at the main result in this section.
\begin{Thm}\label{SPSisINFB} The variety of all strict pseudosemilattices $\mathbf{SPS}$ is inherently non-finitely based.
\end{Thm}
\begin{proof} As  mentioned in the introduction,
$\mathbf{SPS}$ is generated by a finite pseudosemilattice and hence is a locally finite variety.
Let $I$ be a finite set of identities satisfied by $\mathbf{SPS}$. We need to show that there exists an infinite, finitely generated pseudosemilattice satisfying all identities of $I$. By Lemma \ref{nplus1impliesn} and Theorem \ref{basisSPSintermsofwords}
there exists a positive integer $n$ such that all identities in $I$ are consequences of $u_n\approx v_n$. In particular, each pseudosemilattice that satisfies $u_n\approx v_n$ also satisfies all members of $I$. The relatively free pseudosemilattice $\A(x_1,\dots,x_{2n+2})/\rho_n$ clearly satisfies $u_n\approx v_n$ and thus satisfies all identities of $I$; by Corollary \ref{infb} it is infinite and it is clearly finitely generated.
\end{proof}

This result has far reaching consequences. Since for each finite non-associative pseudosemilattice
$E$, the variety $\mathbf{SPS}$ is contained in the variety $\mathsf{var}(E)$ generated by $E$, $\mathsf{var}(E)$ and therefore $E$ itself cannot have a finite identity basis (a result that has been obtained already by the second author by different arguments \cite{lPhD}). In contrast, each associative pseudosemilattice, being an idempotent semigroup, is finitely based. Altogether we have proved the next result.
\begin{Cor} For a finite pseudosemilattice $E$ the following assertions are equivalent:
\begin{enumerate}
\item $E$ is not finitely based,
\item $E$ is inherently non-finitely based,
\item $E$ does not satisfy the associative law.
\end{enumerate}
\end{Cor}
Moreover, $\mathbf{SPS}$ is not contained in any finitely based locally finite variety of \textbf{binary algebras} (not just of pseudosemilattices). Indeed suppose that $\mathbf{SPS}$ were
contained in a finitely based locally finite variety of binary algebras $\mathbf{V}=[s_1\approx t_1,\dots,s_n\approx t_n]$. Then the variety defined by $s_1\approx t_1,\dots,s_n\approx t_n$ together with the identities ((PS1)--(PS3), (PS2'), (PS3'))
would define a finitely based locally finite variety of pseudosemilattices containing $\mathbf{SPS}$, a contradiction.
As a consequence, for example, each variety of the form $\mathbf{SPS}\vee \mathbf{V}$ is not finitely based for any locally finite variety $\mathbf{V}$ of binary algebras. Another  application is as follows. It is one of the basic ingredients \cite{bifreeLI} of the theory of e-varieties of locally inverse semigroups that the sandwich operation $\wedge$ on the set of idempotents $E(S)$ of a (locally inverse) semigroup $S$ admits a canonical extension to a binary operation on $S$ by letting, for $x,y\in S$, $x\wedge y$ be the unique element of $xV(yx)y$ \cite[Lemma 2.1]{trotter} where $V(yx)$ denotes the set of all inverses of $yx$ in $S$ --- the resulting binary algebra is, of course, no longer idempotent. If $S$ is a finite non-$E$-solid locally inverse semigroup then its pseudosemilattice $(E(S),\wedge)$ is inherently non-finitely based and hence so is $(S,\wedge)$ itself (containing $(E(S),\wedge)$ as a substructure).
\begin{Cor} For each finite non-$E$-solid locally inverse semigroup $S$ the binary algebra $(S,\wedge)$ is (inherently) non-finitely based.
\end{Cor}

\section{Supplements and Applications}\label{supplements}

Fix a positive integer $n$ and let $x_1,\dots, x_{2n}$ be pairwise distinct variables; let\hspace*{-.3cm}
\centerline{
\tikz[baseline=11.5pt,scale=0.5]{
\draw (-1.5,1) node{$\be_n'=$};
\coordinate[label=above:$x_{2n}$] (1) at (0,2);
\coordinate[label=below:$x_1$] (2) at (1,0);
\coordinate[label=above:$x_2$] (3) at (2,2);
\coordinate[label=below:$x_3$] (4) at (3,0);
\coordinate[label=above:$x_{2n-2}$] (5) at (5,2);
\coordinate[label=below:$x_{2n-1}$] (6) at (6,0);
\coordinate[label=above:$x_{2n}$] (7) at (7,2);
\coordinate[label=below:$x_{1}$] (8) at (8,0);
\draw (4,1)node{$\dots$};
\draw (1) node {$\bullet$};
\draw (2) node {$\bullet$};
\draw (3) node {$\bullet$};
\draw (4) node {$\bullet$};
\draw (5) node {$\bullet$};
\draw (6) node {$\bullet$};
\draw (7) node {$\bullet$};
\draw (8) node {$\bullet$};
\draw (4)--(3.2,0.7);
\draw (4.8,1.3)--(5);
\draw (1)--(2);
\draw (2)--(3);
\draw (3)--(4);
\draw (5)--(6)--(7);
\draw[double](7)--(8);}
\hspace*{-.1cm}and
\tikz[baseline=11.5pt,scale=0.5]{
\draw (-1.5,1) node{$\de_n=$};
\coordinate[label=below:$x_{1}$] (1) at (0,0);
\coordinate[label=above:$x_2$] (2) at (1,2);
\coordinate[label=below:$x_3$] (3) at (2,0);
\coordinate[label=above:$x_4$] (4) at (3,2);
\coordinate[label=below:$x_{2n-1}$] (5) at (5,0);
\coordinate[label=above:$x_{2n}$] (6) at (6,2);
\coordinate[label=below:$x_{1}$] (7) at (7,0);
\coordinate[label=above:$x_{2}$] (8) at (8,2);
\draw (4,1)node{$\dots$};
\draw (1) node {$\bullet$};
\draw (2) node {$\bullet$};
\draw (3) node {$\bullet$};
\draw (4) node {$\bullet$};
\draw (5) node {$\bullet$};
\draw (6) node {$\bullet$};
\draw (7) node {$\bullet$};
\draw (8) node {$\bullet$};
\draw (4)--(3.2,1.3);
\draw (4.8,0.7)--(5);
\draw[double] (1)--(2);
\draw (2)--(3);
\draw (3)--(4);
\draw (5)--(6)--(7)--(8);}
}
{\parindent=0pt and} let $\al_n'$ be the graph obtained by removing from $\be_n'$ the unique non-distinguished degree 1 vertex. Note that $\al_n$ is obtained from $\de_n$ by removing the unique non-distinguished degree 1 vertex.
\begin{Lemma} For each   $n\ge 2$,
the pairs $(\al_n,\be_n)$ and $(\al_n,\de_n)$ are equivalent.
\end{Lemma}
\begin{proof} Observe that $(\al_n',\be_n')$ can be obtained from $(\al_n,\be_n)$ by changing in $\al_n$ and $\be_n$ the distinguished vertices. It follows from Lemma \ref{changingroots} that $(\al_n',\be_n')$ and $(\al_n,\be_n)$ are equivalent. However, $\de_n$ can be obtained from $\be_n'$ by the automorphism defined by a simple relabeling of the vertices, namely by the transformation $x_1\mapsto x_1$, $x_i\mapsto x_{2n+2-i}$ for all $i\in \{2,3,\dots,2n\}$ (and $x\mapsto x$ for all other $x$) and the same automorphism sends $\al_n'$ to $\al_n$. It follows that $(\al_n',\be_n')$ and $(\al_n,\de_n)$ are also equivalent.
\end{proof}
Inspection of the graph $\al_n$ shows that $\al_n$ is  maximal  in its $\rho_{\mathbf{SPS}}$-class with respect to $\le$. Moreover, $\be_n$ and $\de_n$ are the only elements in the $\rho_{\mathbf{SPS}}$-class of $\al_n$ which are covered by $\al_n$. For any pair $(\al_n,\ga)\in \rho_{\mathbf{SPS}}$ with $\al_n\ne \ga$ we have $\al_n>\al_n\wedge \ga$ whence either $(\al_n,\be_n)$ or $(\al_n,\de_n)$ must be a consequence of $(\al_n,\ga)$. Since $(\al_n,\de_n)$ and $(\al_n,\be_n)$ are equivalent, the pair $(\al_n,\be_n)$ must be a consequence of $(\al_n,\ga)$, anyway. Altogether we arrive at the next Corollary.
\begin{Cor}\label{atomicpair} For each  $n\ge 2$,
the pair $(\al_n,\be_n)$ is a consequence of every pair $(\al_n,\ga)$ for which $\al_n\ne \ga$ and $(\al_n,\ga)\in \rho_{\mathbf{SPS}}$.
\end{Cor}
 We may reformulate Corollary \ref{atomicpair} in terms of identities.
\begin{Cor}\label{atomicidentity} For each  $n\ge 2$,
the identity $u_n\approx v_n$ is a consequence of every non-trivial identity $u_n\approx v$.
\end{Cor}
\begin{proof} We way assume that $\mathbf{SPS}\models u_n\approx v$ for otherwise the latter identity defines, within the variety of all pseudosemilattices, a variety of normal bands  which satisfies $u_n\approx v_n$. The claim then is an immediate consequence of Corollary \ref{atomicpair}.
\end{proof}
\begin{Lemma}\label{singleidentitysuffices} If the identity $u_n\approx v_n$ is a consequence of some set $I$ of identities then $u_n\approx v_n$ is a consequence of a single identity of $I$.
\end{Lemma}
\begin{proof}  Suppose that $u_n\approx v_n$ is a consequence of $I$; then in any derivation process there occurs an identity $u\approx v\in I$ which implies a non-trivial identity $u_n\approx w$. By Corollary \ref{atomicidentity} $u_n\approx v_n$ is a consequence of $u_n\approx w$ and hence also of $u\approx v$.
\end{proof}
We are able to formulate the first major result in this section.
\begin{Thm} \begin{enumerate}
\item If the union $I\cup J$ of two sets of identities is a basis of  $\mathbf{SPS}$ then $I$ or $J$ is already a basis of
$\mathbf{SPS}$.
\item Every co-finite subset of a basis of $\mathbf{SPS}$ is also a basis of $\mathbf{SPS}$.
\item The variety $\mathbf{SPS}$ has no irredundant basis.
\end{enumerate}
\end{Thm}
\begin{proof} (1) Let $I$ and $J$ be two sets of
identities such that neither is a basis of $\mathbf{SPS}$; then there are $k,n$ such that $u_k\approx v_k$  is not a consequence of $I$ and $u_n\approx v_n$ is not a consequence of $J$. Then by Lemmas \ref{nplus1impliesn}
and \ref{singleidentitysuffices}, $u_l\approx v_l$ is not a consequence of $I\cup J$ for $l=\max(k,n)$.
(2) and (3)  are  immediate consequences of (1) since $\mathbf{SPS}$ does not have a finite basis by Theorem \ref{SPSisINFB}.\end{proof}

An element $a$ of a lattice $\mathcal{L}$ is \emph{$\wedge$-irreducible} if for any $b,c\in \mathcal{L}$, the equality  $a=b\wedge c$ implies $a=b$ or $a=c$; $a$ is \emph{$\wedge$-prime} if  $a\ge b\wedge c$ implies $a\ge b$ or $a\ge c$. It is well known that every $\wedge$-prime element is also $\wedge$-irreducible.
\begin{Prop}\label{covers}
The variety $\mathbf{SPS}$ is $\cap$-pri\-me and $\cap$-irreducible in the lattice $\mathcal{L}(\mathbf{PS})$. Further, $\mathbf{SPS}$ has no covers in $\mathcal{L}(\mathbf{PS})$.
\end{Prop}
\begin{proof}
 As already mentioned, the lattice $\mathcal{L}(\mathbf{PS})$ is the disjoint union of the intervals $[\mathbf{T},\mathbf{NB}]$ and $[\mathbf{SPS},\mathbf{PS}]$ with $\mathbf{NB}\subsetneq \mathbf{SPS}$, where $\mathbf{T}$ is the variety of all trivial binary algebras. Thus, for $\mathbf{U},\,\mathbf{V}\in\mathcal{L}(\mathbf{PS})$,
\[\mathbf{U}\cap \mathbf{V}\subsetneq \mathbf{SPS}\qquad\mbox{ implies }\qquad \mathbf{U}\subseteq \mathbf{ NB}\quad \mbox{ or }\quad\mathbf{V}\subseteq \mathbf{ NB}.\;\]
Assume next that $\mathbf{SPS}=\mathbf{U}\cap \mathbf{V}$ and let $I$ and $J$ be identity bases of $\mathbf{U}$ and $\mathbf{V}$, respectively. Then $I\cup J$  is an identity basis  for $\mathbf{SPS}$, and so $I$ or $J$ is also a basis  for $\mathbf{SPS}$ by the previous result. Thus $\mathbf{U}=\mathbf{SPS}$ or $\mathbf{V}=\mathbf{SPS}$, and we have shown that $\mathbf{SPS}$ is $\cap$-prime and $\cap$-irreducible in $\mathcal{L}(\mathbf{PS})$.

Let now $\mathbf{U}$ be a cover of $\mathbf{SPS}$. Then there exists an identity $u_n\approx v_n$  which does not hold in $\mathbf{U}$. Let $\mathbf{V}$ be the variety of \pseudo s defined by  $u_n\approx v_n$. Then $\mathbf{U}\not\subseteq \mathbf{V}\,$ and $\,\mathbf{SPS}=\mathbf{U}\cap \mathbf{V}$, which contradicts the fact that $\mathbf{SPS}$ is $\cap$-irreducible since $\mathbf{U}\ne\mathbf{SPS}\ne\mathbf{V}$.
\end{proof}
 Finally, let us briefly mention some applications to the lattice of e-varieties of locally inverse semigroups. Let $E\mathbf{SR}$ be the e-variety of all locally inverse semigroups $S$ whose idempotent generated subsemigroup $\left<E(S)\right>$ is strict regular. Then  $\mathbf{SPS}\varphi^{-1}=[\mathbf{CSR},E\mathbf{SR}]$  for the complete \hyphenation{homo-morphism} homomorphism $\varphi$ of (\ref{sgtops}). As a consequence of Proposition \ref{covers} we get:
 \begin{Cor} The e-variety $E\mathbf{SR}$ is $\cap$-prime  and has no covers in the lattice ${\mathcal{L}}_e(\mathbf{ LI})$ of all e-varieties of locally inverse semigroups.
 \end{Cor}
 A concept of identities --- \emph{bi-identities} --- has been introduced for e-varieties of locally inverse semigroups by the first author \cite{bifreeLI,biequationaltheory},  subject to which the notion of
an identity basis makes sense (the reader is referred to the cited papers for more details). We note that the following result is not covered by Theorem 4.2 in \cite{adv2} which presents a sufficient condition for an e-variety of locally inverse semigroups to be not finitely based.
 \begin{Cor} The e-variety $E\mathbf{SR}$ has no finite basis for its bi-identities.
 \end{Cor}
 \begin{proof} The set of bi-identities
 $$\{u_n(x_1x_1',\dots,x_{2n}x_{2n}')\approx v_n(x_1x_1',\dots,x_{2n}x_{2n}')\mid n\ge 2\}$$
 is an infinite basis for the bi-identities of $E\mathbf{SR}$ by Theorem \ref{basisSPSintermsofwords}. Theorem \ref{SPSisINFB} shows that no finite subset is a basis. The claim then follows from the Compactness Theorem of Equational Logic which also holds in the context of bi-identities.
 \end{proof}

\noindent\textbf{Acknowledgments}: The second author was partially supported by the European Regional Development Fund through the programme COMPETE and by the Portuguese Government through the FCT -- Funda\c{c}\~ao para a Ci\^encia e a Tecnologia under the project PEst-C/MAT/UI0144/2011.


\begin{thebibliography}{99}


\bibitem{wpstrict}
    K.~Auinger, The word problem for the bifree combinatorial strict regular semigroup, {\em Math. Proc. Cambridge Philos. Soc.\/} {\bf 113} (1993), 519--533.

\bibitem{au4}
    K.~Auinger, On the lattice of existence varieties of locally inverse semigroups, {\em Canad. Math. Bull.\/} {\bf 37} (1994), 13--20.

\bibitem{aufreepseudo}
    K.~Auinger, The free pseudo-semilattice on a set (English summary), {\it Contributions to general algebra}, {\bf 9} (Linz, 1994), 37--48, H\"older-Pichler-Tempsky, Vienna, 1995.
\bibitem{bifreeLI} K.~Auinger, {The bifree locally inverse semigroup on a set}, \emph{J.~Algebra} \textbf{166} (1994), 630--650.
\bibitem{biequationaltheory} K.~Auinger, {A system of bi-identities for locally inverse semigroups}, \emph{Proc. Amer. Math. Soc.} \textbf{123} (1995), 979--988.

\bibitem{adv2} K.~Auinger, I.~Dolinka, M.~V.~Volkov, {Equational theories of semigroups with involution}, \emph{J.~Algebra} \textbf{369} (2012), 203--225.

\bibitem{ha1}
    T.~E.~Hall, Identities for existence varieties of regular semigroups, {\em Bull. Austral. Math. Soc.\/} {\bf 40} (1989), 59--77.
\bibitem{hallCSR} T.~E.~Hall, Regular semigroups: amalgamation and the lattice of existence varieties, \emph{Algebra Universalis} \textbf{28} (1991), 79--102.

\bibitem{ks1}
    J.~Ka\softd ourek and M.~B.~Szendrei, A new approach in the theory of orthodox semigroups, {\em Semigroup Forum\/} {\bf 40} (1990), 257--296.

\bibitem{mea2}
    J.~Meakin, The free local semilattice on a set, {\em J. Pure Appl. Algebra\/} {\bf 27} (1983), 263--275.

\bibitem{mp2}
    J.~Meakin and F.~Pastijn, The free pseudo-semilattice on two generators, {\em Algebra Universalis\/} {\bf 14} (1982), 297--309.

\bibitem{na1}
   K.~S.~S.~Nambooripad, Pseudo-semilattices and biordered sets I, {\em Simon  Stevin\/} {\bf 55} (1981), 103--110.

\bibitem{lPhD}
		L.~Oliveira, \emph{Varieties of pseudosemilattices}, Ph.D. dissertation, Marquette University, 2004.

\bibitem{l1}
    L.~Oliveira, A solution to the word problem for free pseudosemilattices, {\em Semigroup Forum\/} {\bf 68} (2004), 246--267.

\bibitem{l2}
    L.~Oliveira, Models for free pseudosemilattices, {\em Algebra Universalis\/} {\bf 56} (2007), 315--336.

\bibitem{trees} J.-P.~Serre, \emph{Trees}, Springer, Berlin Heidelberg 1980.

\bibitem{trotter} P.~G.~Trotter, Congruence extensions in regular semigroups, \emph{J. Algebra} \textbf{137} (1991), 166--179.
\end{thebibliography}
\end{document}